\documentclass[12pt,reqno]{amsart}
\usepackage{amsmath,amsfonts,amssymb,amsthm,amsbsy,latexsym,cite,amscd}
\usepackage[cp1251]{inputenc}
\usepackage[%
  colorlinks,
  breaklinks,
  linkcolor=blue,
  citecolor=blue,
  menucolor=blue,
  pagecolor=blue,
]{hyperref}

\multlinegap=\parindent

\begin{document}

\newtheorem{rtheorem}{Теорема}[section]
\newtheorem{rcorollary}[rtheorem]{Следствие}
\newtheorem{rproposition}[rtheorem]{Предложение}
\newtheorem{etheorem}{Theorem}[section]
\newtheorem{ecorollary}[etheorem]{Corollary}
\newtheorem{eproposition}[etheorem]{Proposition}

\renewcommand{\theequation}{$\ast$}
\renewcommand{\thefootnote}{}

\title[On~the~separability of~subgroups of~nilpotent groups]{On~the~separability of~subgroups of~nilpotent groups by~root classes of~groups}

\author{E.~V.~Sokolov}
\address{Ivanovo State University, Russia}
\email{ev-sokolov@yandex.ru}

\begin{abstract}
Suppose that $\mathcal{C}$ is a~class of~groups consisting only of~periodic groups and~$\mathfrak{P}(\mathcal{C})^{\prime}$ is the~set of~prime numbers each of~which does not divide the~order of~any element of~a~$\mathcal{C}$\nobreakdash-group. It~is easy to~see that if a~subgroup~$Y$ of~a~group~$X$ is $\mathcal{C}$\nobreakdash-sepa\-ra\-ble in~this group, then it is $\mathfrak{P}(\mathcal{C})^{\prime}$\nobreakdash-iso\-lat\-ed in~$X$. Let us say that $X$ has the~property~$\mathcal{C}\mbox{-}\mathfrak{Sep}$ if all its $\mathfrak{P}(\mathcal{C})^{\prime}$\nobreakdash-iso\-lat\-ed subgroups are $\mathcal{C}$\nobreakdash-sepa\-ra\-ble. We find a~condition that is sufficient for~a~nilpotent group~$N$ to~have the~property~$\mathcal{C}\mbox{-}\mathfrak{Sep}$ provided $\mathcal{C}$ is a~root class. We also prove that if $N$ is tor\-sion-free, then the~indicated condition is necessary for~this group to~have~$\mathcal{C}\mbox{-}\mathfrak{Sep}$.
\end{abstract}

\keywords{Subgroup separability, residual finiteness, residual $p$\nobreakdash-finiteness, residual solvability, root-class residuality, nilpotent group, residually nilpotent group, generalized free product, tree product, HNN-extension, fundamental group of~a~graph of~groups}

\thanks{The~study was supported by~the~Russian Science Foundation grant No.~22-21-00166,\\ \url{https://rscf.ru/en/project/22-21-00166/}}

\maketitle\vspace{-25pt}

\section{Introduction}\label{se01}

Let us begin with~the~definitions of~the~concepts appearing in~the~title of~the~article. A~subgroup~$Y$ of~a~group~$X$ is said to~be \emph{separable in~$X$ by~a~class of~groups~$\mathcal{C}$} (\emph{$\mathcal{C}$\nobreakdash-sepa\-ra\-ble} for~brevity) if, for~any element $x \in X \setminus Y$, there exists a~homomorphism~$\sigma$ of~$X$ onto~a~group from~$\mathcal{C}$ such that $x\sigma \notin Y\sigma$~\cite{Malcev1958PIvPI}. If~the~trivial subgroup of~$X$ is $\mathcal{C}$\nobreakdash-sepa\-ra\-ble, then $X$ is called \emph{residually a~$\mathcal{C}$\nobreakdash-group}.

The~concept of~a~\emph{root class of~groups} has several equivalent definitions~\cite{Sokolov2015CA}. According to~one of~them, a~class of~groups~$\mathcal{C}$ is called root if it contains at~least one non-trivial group and~is closed under taking subgroups, extensions, and~Cartesian products of~the~form $\prod_{y \in Y}X_{y}$, where $X, Y \in \mathcal{C}$ and~$X_{y}$ is an~isomorphic copy of~$X$ for~each $y \in Y$. The~examples of~root classes are the~classes of~all finite groups, finite $p$\nobreakdash-groups (where $p$ is a~prime number), periodic $\mathfrak{P}$\nobreakdash-groups of~finite exponent (where $\mathfrak{P}$ is a~non-empty set of~primes), all solvable groups, and~all tor\-sion-free groups. It~is also easy to~show that the~intersection of~any number of~root classes is again a~root class.

The~notion of~a~root class was introduced in~\cite{Gruenberg1957PLMS} and~allows one to~prove many statements at~once using the~same reasoning. It~turns out to~be especially useful in~studying the~residual properties of~free constructions of~groups (see, e.g.,~\cite{Sokolov2015CA, Sokolov2017SMJ, Sokolov2021SMJ2, Sokolov2022CA, SokolovTumanova2016SMJ, SokolovTumanova2017MZ, SokolovTumanova2019AL, SokolovTumanova2020IVM, Tumanova2017SMJ, Tumanova2019SMJ}). If~$X$ is such a~construction and~$\mathcal{C}$ is a~root class of~groups, then the~$\mathcal{C}$\nobreakdash-sepa\-ra\-bil\-ity of~some subgroups of~$X$ is quite often one of~the~necessary and/or~sufficient conditions for~$X$ to~be residually a~$\mathcal{C}$\nobreakdash-group. The~examples of~assertions of~this type can be found, in~particular, in~\cite{Azarov2013SMJ1, Azarov2013MZ, Azarov2016SMJ, Berlai2016CA, Logan2018CA, Loginova1999SMJ, Sokolov2021SMJ2, Sokolov2022CA, SokolovTumanova2016SMJ, SokolovTumanova2020IVM}. These results become more constructive if the~description of~$\mathcal{C}$\nobreakdash-sepa\-ra\-ble subgroups is given for~the~groups from~which the~construction is~\mbox{composed}.

Most of~the~known facts on~the~separability of~subgroups concern the~property of~finite separability, i.e.,~separability by~the~class of~all finite groups. There is also a~series of~assertions on~the~separability by~the~class of~finite $\mathfrak{P}$\nobreakdash-groups, where $\mathfrak{P}$ is some (most often one-element) set of~primes (see~\cite{Bardakov2004SMJ, Bardakov2006AC, BerlaiFerov2019JA, BobrovskiiSokolov2010AC, Loginova1999SMJ, Moldavanskii2017BIvSU, Sokolov2014IJAC, Sokolov2017SMJ}). The~separability of~subgroups by~other classes of~groups is not actually studied.

In~this article, we consider the~separability of~subgroups of~nilpotent groups by~a~root class of~groups~$\mathcal{C}$ and~give examples of~using the~obtained results in~studying the~property of~being residually a~$\mathcal{C}$\nobreakdash-group (in~the~case of~free constructions of~groups). It~follows from~Proposition~\ref{pe403} below that if $\mathcal{C}$ contains at~least one non-periodic group and~is closed under taking quotient groups, then it also contains all nilpotent groups of~cardinality less than~$\aleph_{\infty}$. It~is clear that all subgroups of~such groups are automatically $\mathcal{C}$\nobreakdash-sepa\-ra\-ble. So in~what follows, we consider only the~case when the~class~$\mathcal{C}$ consists of~periodic groups.{\parfillskip=0pt\par}

Everywhere below, if $\mathcal{C}$ is an~arbitrary (not~necessarily root) class of~periodic groups, then we denote by~$\mathfrak{P}(\mathcal{C})$ the~set of~all the~prime numbers each of~which divides the~order of~an~element of~some $\mathcal{C}$\nobreakdash-group. It~turns out that $\mathfrak{P}(\mathcal{C})$ plays an~important role in~the~study of~$\mathcal{C}$\nobreakdash-sepa\-ra\-bil\-ity. To~clarify this relationship, we need some notation and~definitions.

Throughout the~paper, if $\mathfrak{P}$ is a~set of~primes, then $\mathfrak{P}^{\prime}$ denotes the~set of~all the~prime numbers that do not belong to~$\mathfrak{P}$. A~subgroup~$Y$ of~a~group~$X$ is called \emph{$\mathfrak{P}^{\prime}$\nobreakdash-iso\-lat\-ed} in~this group if, for~any $x \in X$, $q \in \mathfrak{P}^{\prime}$, it follows from~the~inclusion $x^{q} \in Y$ that $x \in Y$. If~the~trivial subgroup of~$X$ is $\mathfrak{P}^{\prime}$\nobreakdash-iso\-lat\-ed, then $X$ is said to~\emph{have no~$\mathfrak{P}^{\prime}$\nobreakdash-tor\-sion}. We note that if $\mathfrak{P}$ contains all primes, then every subgroup is $\mathfrak{P}^{\prime}$\nobreakdash-iso\-lat\-ed.

It is known (see Proposition~\ref{pe405} below) that if $\mathcal{C}$ is a~class of~groups consisting only of~periodic groups, then any $\mathcal{C}$\nobreakdash-sepa\-ra\-ble subgroup is $\mathfrak{P}(\mathcal{C})^{\prime}$\nobreakdash-iso\-lat\-ed. We say that a~group~$X$ \emph{has the~property~$\mathcal{C}\mbox{-}\mathfrak{Sep}$} if the~inverse statement also holds, i.e.,~all the~$\mathfrak{P}(\mathcal{C})^{\prime}$\nobreakdash-iso\-lat\-ed subgroups of~$X$ are $\mathcal{C}$\nobreakdash-sepa\-ra\-ble. In~this article, we study the~following question: what conditions are sufficient for~a~nilpotent group to~have the~property~$\mathcal{C}\mbox{-}\mathfrak{Sep}$ if $\mathcal{C}$ is a~root class of~groups consisting only of~periodic groups. The~result of~this research is the~notion of~a~$\mathcal{C}$\nobreakdash-bound\-ed nilpotent group, its definition is given in~the~next section.

\section{Main results}\label{se02}

Suppose that $\mathcal{C}$ is a~class of~groups consisting only of~periodic groups and~$A$ is an~abelian group. If~$p \in \mathfrak{P}(\mathcal{C})$, then we refer to~the~$p$\nobreakdash-power torsion subgroup of~$A$ as~the~\emph{primary $\mathfrak{P}(\mathcal{C})$\nobreakdash-com\-po\-nent} of~this group. Let us say that~$A$

--\hspace{1ex}is \emph{weakly $\mathcal{C}$\nobreakdash-bound\-ed} if, for~any quotient group~$B$ of~$A$, all the~primary $\mathfrak{P}(\mathcal{C})$\nobreakdash-com\-po\-nents of~$B$ are of~finite exponent;

--\hspace{1ex}is \emph{$\mathcal{C}$\nobreakdash-bound\-ed} if, for~any quotient group~$B$ of~$A$, each primary $\mathfrak{P}(\mathcal{C})$\nobreakdash-com\-po\-nent of~$B$ has a~finite exponent and~a~cardinality not~exceeding the~cardinality of~some $\mathcal{C}$\nobreakdash-group.

We call a~nilpotent group (\emph{weakly}) \emph{$\mathcal{C}$\nobreakdash-bounded} if it has at~least one finite central series with~(weakly) $\mathcal{C}$\nobreakdash-bound\-ed abelian factors. The~classes of~$\mathcal{C}$\nobreakdash-bound\-ed abelian, $\mathcal{C}$\nobreakdash-bound\-ed nilpotent, weakly $\mathcal{C}$\nobreakdash-bound\-ed abelian, and~weakly $\mathcal{C}$\nobreakdash-bound\-ed nilpotent groups are denoted below by~$\mathcal{C}\mbox{-}\mathcal{BA}$, $\mathcal{C}\mbox{-}\mathcal{BN}$, $\mathcal{C}\mbox{-}w\mathcal{BA}$, and~$\mathcal{C}\mbox{-}w\mathcal{BN}$ respectively.

Let us note that if the~class~$\mathcal{C}$ is root, then a~finite $\mathcal{C}$\nobreakdash-group can have an~arbitrarily large order (see Proposition~\ref{pe401} below) and~therefore any finitely generated nilpotent group is $\mathcal{C}$\nobreakdash-bound\-ed. It~also follows from~Proposition~\ref{pe403} that a~weakly $\mathcal{C}$\nobreakdash-bound\-ed nilpotent group~$X$ is $\mathcal{C}$\nobreakdash-bound\-ed if $\mathcal{C}$ contains infinite groups and~the~cardinality of~$X$ is less than~$\aleph_{\infty}$.

The~first result of~this article is Theorem~\ref{te201} that describes the~abelian groups having the~property~$\mathcal{C}\mbox{-}\mathfrak{Sep}$.

\begin{etheorem}\label{te201}
If $\mathcal{C}$ is a~root class of~groups consisting only of~periodic groups and~$X$ is an~abelian group, then the~following statements hold.

\textup{1.}\hspace{1ex}If $X$ has the~property~$\mathcal{C}\mbox{-}\mathfrak{Sep}$, then it is weakly $\mathcal{C}$\nobreakdash-bound\-ed.

\textup{2.}\hspace{1ex}If $X$ is weakly $\mathcal{C}$\nobreakdash-bound\-ed, then it has the~properties $\mathcal{C}\mbox{-}\mathfrak{Sep}$ and~$\mathcal{P}\mbox{-}\mathfrak{Sep}$, where
$$
\mathcal{P} = \bigcup_{p \in \mathfrak{P}(\mathcal{C})} \mathcal{F}_{p}
$$
and~$\mathcal{F}_{p}$ is the~class of~finite $p$\nobreakdash-groups.
\end{etheorem}

Thus, an~abelian group has the~property~$\mathcal{C}\mbox{-}\mathfrak{Sep}$ if and~only if it is weakly $\mathcal{C}$\nobreakdash-bound\-ed. However, this property may not hold for~weakly $\mathcal{C}$\nobreakdash-bound\-ed nilpotent groups: the~corresponding example is constructed in~Section~\ref{se09}. The~main result of~this article is Theorem~\ref{te202} given below, which says that any $\mathcal{C}$\nobreakdash-bound\-ed nilpotent group enjoys $\mathcal{C}\mbox{-}\mathfrak{Sep}$. In~fact, this theorem states somewhat more, so its formulation must be preceded by~several definitions.

Suppose again that $\mathfrak{P}$ is an~arbitrary set of~primes, $X$ is a~group, and~$Y$ is a~subgroup of~$X$. It~is easy to~see that the~intersection of~any number of~$\mathfrak{P}^{\prime}$\nobreakdash-iso\-lat\-ed subgroups of~$X$ is in~turn a~$\mathfrak{P}^{\prime}$\nobreakdash-iso\-lat\-ed subgroup and~therefore there exists the~smallest $\mathfrak{P}^{\prime}$\nobreakdash-iso\-lat\-ed subgroup containing~$Y$. It~is called the~\emph{$\mathfrak{P}^{\prime}$\nobreakdash-iso\-la\-tor} of~$Y$ in~$X$, and~we denote it by~$\mathfrak{P}^{\prime}\mbox{-}\mathfrak{Is}(X, Y)$. We also denote by~$\mathfrak{P}^{\prime}\mbox{-}\mathfrak{Rt}(X, Y)$ the~\emph{set of~$\mathfrak{P}^{\prime}$\nobreakdash-roots} extracted in~$X$ from~the~elements of~$Y$. More accurately, an~element $x \in X$ belongs to~$\mathfrak{P}^{\prime}\mbox{-}\mathfrak{Rt}(X, Y)$ if there exists a~$\mathfrak{P}^{\prime}$\nobreakdash-num\-ber~$q$ such that $x^{q} \in Y$. Obviously, the~set $\mathfrak{P}^{\prime}\mbox{-}\mathfrak{Rt}(X, Y)$ is contained in~the~subgroup $\mathfrak{P}^{\prime}\mbox{-}\mathfrak{Is}(X, Y)$ and~coincides with~the~latter if and~only if it is itself a~subgroup.

\begin{etheorem}\label{te202}
Suppose that $\mathcal{C}$ is a~root class of~groups consisting only of~periodic groups, $X$ is a~$\mathcal{C}\mbox{-}\mathcal{BN}$\nobreakdash-group, and~$Y$ is a~subgroup of~$X$. Then the~$\mathfrak{P}(\mathcal{C})^{\prime}$\nobreakdash-iso\-la\-tor $\mathfrak{P}(\mathcal{C})^{\prime}\mbox{-}\mathfrak{Is}(X, Y)$ is $\mathcal{C}$\nobreakdash-sepa\-ra\-ble in~$X$ and~coincides with~the~set $\mathfrak{P}(\mathcal{C})^{\prime}\mbox{-}\mathfrak{Rt}(X, Y)$. In~particular, $X$ has the~property~$\mathcal{C}\mbox{-}\mathfrak{Sep}$. If~$X$ is $\mathfrak{P}(\mathcal{C})^{\prime}$\nobreakdash-tor\-sion-free, then $Y$ and~$\mathfrak{P}(\mathcal{C})^{\prime}\mbox{-}\mathfrak{Is}(X, Y)$ have the~same nilpotency classes.
\end{etheorem}

We note that the~requirement for~the~group~$X$ from~Theorem~\ref{te202} to~be $\mathfrak{P}(\mathcal{C})^{\prime}$\nobreakdash-tor\-sion-free is essential for~the~equality of~the~nilpotency classes of~$Y$ and~$\mathfrak{P}(\mathcal{C})^{\prime}\mbox{-}\mathfrak{Is}(X, Y)$. For~example, if $Y$ is an~infinite cyclic group, $Z$ is a~finite nilpotent $\mathfrak{P}(\mathcal{C})^{\prime}$\nobreakdash-group of~class~$c$, and~$X = Y \times Z$, then $\mathfrak{P}(\mathcal{C})^{\prime}\mbox{-}\mathfrak{Is}(X, Y) = X$ and~therefore the~nilpotency class of~$\mathfrak{P}(\mathcal{C})^{\prime}\mbox{-}\mathfrak{Is}(X, Y)$ is equal to~$c$, while $Y$ is a~group of~class~$1$.

Theorem~\ref{te201} shows that an~abelian group with~the~property~$\mathcal{C}\mbox{-}\mathfrak{Sep}$ need not be $\mathcal{C}$\nobreakdash-bound\-ed. At~the~same time, the~following assertion holds.

\begin{ecorollary}\label{ce203}
Let $\mathcal{C}$ be a~root class of~groups consisting only of~periodic groups. A~tor\-sion-free nilpotent group has the~property~$\mathcal{C}\mbox{-}\mathfrak{Sep}$ if and~only if it is $\mathcal{C}$\nobreakdash-bound\-ed.
\end{ecorollary}

It turns out that the~analogues of~Theorem~\ref{te202} and~Corollary~\ref{ce203} do not hold even for~supersolvable groups. For~example, if $\mathcal{C}$ is a~root class of~groups consisting only of~periodic groups, then, by~the~main theorem from~\cite{Tumanova2017SMJ} and~Proposition~\ref{pe807} given below, the~supersolvable Baumslag--Solitar group
$$
\mathit{BS}(1, -1) = \langle a, b;\ a^{-1}ba = b^{-1} \rangle
$$
is residually a~$\mathcal{C}$\nobreakdash-group if and~only if $2 \in \mathfrak{P}(\mathcal{C})$. At~the~same time, $\mathit{BS}(1, -1)$ is tor\-sion-free and~therefore its trivial subgroup is always $\mathfrak{P}(\mathcal{C})^{\prime}$\nobreakdash-iso\-lat\-ed.

Throughout the~paper, if $\mathcal{C}$ is a~root class of~groups consisting only of~periodic groups, then we denote by~$\mathcal{C}\mbox{-}\mathcal{BN}_{\mathfrak{P}(\mathcal{C})}$ the~class of~all $\mathfrak{P}(\mathcal{C})^{\prime}$\nobreakdash-tor\-sion-free $\mathcal{C}\mbox{-}\mathcal{BN}$\nobreakdash-groups. This class arises quite naturally because, by~Theorem~\ref{te202}, only such a~$\mathcal{C}\mbox{-}\mathcal{BN}$\nobreakdash-group is residually a~$\mathcal{C}$\nobreakdash-group. The~theorem given below serves as~a~partial generalization of~Theorem~\ref{te202}.{\parfillskip=0pt\par}

\begin{etheorem}\label{te204}
Suppose that $\mathcal{C}$ is a~root class of~groups consisting only of~periodic groups and~$\mathcal{NC}$ is a~class of~nilpotent $\mathcal{C}$\nobreakdash-groups. Suppose also that $X$ is residually a~$\mathcal{C}\mbox{-}\mathcal{BN}_{\mathfrak{P}(\mathcal{C})}$-group, $Y$ is a~subgroup of~$X$, and~there exists a~homomorphism~$\sigma$ of~$X$ onto~a~$\mathcal{C}\mbox{-}\mathcal{BN}_{\mathfrak{P}(\mathcal{C})}$-group that acts injectively on~$Y$. Then the~following statements hold.

\textup{1.}\hspace{1ex}The~$\mathfrak{P}(\mathcal{C})^{\prime}$\nobreakdash-iso\-la\-tor $\mathfrak{P}(\mathcal{C})^{\prime}\mbox{-}\mathfrak{Is}(X, Y)$ is $\mathcal{NC}$\nobreakdash-sepa\-ra\-ble in~$X$, coincides with~the~set $\mathfrak{P}(\mathcal{C})^{\prime}\mbox{-}\mathfrak{Rt}(X, Y)$, and~has the~same nilpotency class as~$Y$. The~group~$X$ is residually an~$\mathcal{NC}$\nobreakdash-group and~therefore has no~$\mathfrak{P}(\mathcal{C})^{\prime}$\nobreakdash-tor\-sion.

\textup{2.}\hspace{1ex}The~homomorphism~$\sigma$ acts injectively on~the~subgroup $\mathfrak{P}(\mathcal{C})^{\prime}\mbox{-}\mathfrak{Is}(X, Y)$; as~a~result, the~latter belongs to~$\mathcal{C}\mbox{-}\mathcal{BN}_{\mathfrak{P}(\mathcal{C})}$.
\end{etheorem}

We note that if $X \in \mathcal{C}\mbox{-}\mathcal{BN}$, then any homomorphic image of~$X$ is nilpotent. Therefore, Theorem~\ref{te202} actually states that the~subgroup $\mathfrak{P}(\mathcal{C})^{\prime}\mbox{-}\mathfrak{Is}(X, Y)$ is separable in~$X$ by~the~class~$\mathcal{NC}$ of~all nilpotent $\mathcal{C}$\nobreakdash-groups. A~similar argument does not work for~residually $\mathcal{C}\mbox{-}\mathcal{BN}_{\mathfrak{P}(\mathcal{C})}$\nobreakdash-groups, and,~as~a~result, the~class $\mathcal{NC}$ appears in~Theorem~\ref{te204} explicitly.{\parfillskip=0pt\par}

We also note that it is essential for~the~$\mathcal{C}$\nobreakdash-sepa\-ra\-bil\-ity of~$\mathfrak{P}(\mathcal{C})^{\prime}\mbox{-}\mathfrak{Is}(X, Y)$ that there exists a~homomorphism of~$X$ onto~a~$\mathcal{C}\mbox{-}\mathcal{BN}_{\mathfrak{P}(\mathcal{C})}$\nobreakdash-group acting injectively on~$Y$. For~example, any non-abelian free group is residually a~finitely generated tor\-sion-free nilpotent group~\cite{Magnus1935MA, Hall1950PAMS}, but it also contains a~finitely generated isolated subgroup that is not separable by~the~class~$\mathcal{F}_{p}$ of~all finite $p$\nobreakdash-groups for~any prime number~$p$~\cite{Bardakov2004SMJ}.

\section{Some applications}\label{se03}

In~this section, we formulate a~number of~assertions which concern the~residual properties of~free constructions of~groups and~can be deduced from~known results using Theorem~\ref{te204} and~some other properties of~$\mathcal{C}\mbox{-}\mathcal{BN}$\nobreakdash-groups. The~free construction under consideration is always composed of~residually $\mathcal{C}$\nobreakdash-groups, and~Theorem~\ref{te204} completely generalizes Theorem~\ref{te202} if they are applied to~such groups. Therefore, it makes no~sense to~formulate separate corollaries from~Theorem~\ref{te202}.

Throughout this section, we assume that $\mathcal{C}$ is a~root class of~groups consisting only of~periodic groups,
$$
P = \langle A * B;\ U \rangle
$$
is the~free product of~groups~$A$ and~$B$ with~an~amalgamated subgroup~$U$,~and
$$
G^{*} = \langle G, t;\ t^{-1}Ht = K,\ \varphi \rangle
$$
is the~HNN-ex\-ten\-sion of~a~group~$G$ with~subgroups~$H$ and~$K$ associated by~an~isomorphism $\varphi\colon H \to K$ (the~definitions of~these constructions can be found, e.g.,~in~\cite{LyndonSchupp1977}). It~is also assumed that $A \ne U \ne B$ and~$H \ne G \ne K$.

\begin{etheorem}\label{te301}
Suppose that $U$ is a~retract of~$B$, $A$ is residually a~$\mathcal{C}\mbox{-}\mathcal{BN}_{\mathfrak{P}(\mathcal{C})}$\nobreakdash-group, and~there exists a~homomorphism of~$A$ onto~a~$\mathcal{C}\mbox{-}\mathcal{BN}_{\mathfrak{P}(\mathcal{C})}$\nobreakdash-group that acts injectively on~$U$. Then $P$ is residually a~$\mathcal{C}$\nobreakdash-group if and~only if $B$ has the~same property.
\end{etheorem}

If $X$ is a~group and~$Y$ is a~normal subgroup of~$X$, then the~restrictions on~$Y$ of~all inner automorphisms of~$X$ compose a~subgroup of~the~automorphism group~$\operatorname{Aut}Y$, which we denote by~$\operatorname{Aut}_{X}(Y)$. Obviously, if $U$ is normal in~$A$ and~$B$, then it is normal in~$P$ and~the~group $\operatorname{Aut}_{P}(U)$ is generated by~its subgroups $\operatorname{Aut}_{A}(U)$ and~$\operatorname{Aut}_{B}(U)$.

\begin{etheorem}\label{te302}
Suppose that $U$ is normal in~$A$ and~$B$, $\operatorname{Aut}_{P}(U)$ is abelian or~coincides with~$\operatorname{Aut}_{A}(U)$ or~$\operatorname{Aut}_{B}(U)$. Suppose also that $A$ and~$B$ are residually $\mathcal{C}\mbox{-}\mathcal{BN}_{\mathfrak{P}(\mathcal{C})}$\nobreakdash-groups and~have homomorphisms onto~$\mathcal{C}\mbox{-}\mathcal{BN}_{\mathfrak{P}(\mathcal{C})}$\nobreakdash-groups that act injectively on~$U$. Then $P$ is residually a~$\mathcal{C}$\nobreakdash-group if and~only if $U$ is $\mathfrak{P}(\mathcal{C})^{\prime}$\nobreakdash-iso\-lat\-ed in~$A$ and~$B$.
\end{etheorem}

\begin{etheorem}\label{te303}
Suppose that $K_{0}\kern-1.5pt{} =\kern-1.5pt{} G$\kern-1pt{}, $H_{1}\kern-1.5pt{} =\kern-1.5pt{} H$\kern-1.5pt{}, $K_{1}\kern-1.5pt{} =\kern-1.5pt{} K$\kern-1pt{}, and~$H_{i+1}\kern-1.5pt{} =\kern-1.5pt{} H_{i} \cap K_{i}$,~$K_{i+1}\kern-1.5pt{} =\nolinebreak\kern-1.5pt{} H_{i+1}\varphi$ for~all $i \geqslant 1$. Suppose also that $H$ and~$K$ lie in~the~center of~$G$ and~there exists $m \geqslant\nolinebreak 1$ such that $H_{m}$ and~$K_{m}$ are finitely generated. If~$G$ is residually a~$\mathcal{C}\mbox{-}\mathcal{BN}_{\mathfrak{P}(\mathcal{C})}$\nobreakdash-group and~has a~homomorphism~$\sigma$ onto~a~$\mathcal{C}\mbox{-}\mathcal{BN}_{\mathfrak{P}(\mathcal{C})}$\nobreakdash-group that acts injectively on~$HK$, then $G^{*}$ is residually a~$\mathcal{C}$\nobreakdash-group if and~only~if

\textup{1)}\hspace{1ex}$H_{n} = K_{n}$ for~some $n > m$;

\textup{2)}\hspace{1ex}$H$ and~$K$ are $\mathfrak{P}(\mathcal{C})^{\prime}$\nobreakdash-iso\-lat\-ed~in~$G$;

\textup{3)}\hspace{1ex}$\bigcap_{N \in \Omega} N = 1$, where $\Omega$ is the~family of~subgroups of~$H_{n}$ defined as~follows: $N \in \Omega$ if and~only if $H_{n}/N$ is a~finite $\mathfrak{P}(\mathcal{C})$\nobreakdash-group, $N\varphi = N$, and~the~automorphism of~$H_{n}/N$ induced by~$\varphi$ has the~order which is a~$\mathfrak{P}(\mathcal{C})$\nobreakdash-num\-ber.
\end{etheorem}

If $H$ and~$K$ are infinite cyclic subgroups, Theorem~\ref{te303} acquires the~following, simpler formulation.

\begin{etheorem}\label{te304}
Suppose that $H$ and~$K$ are infinite cyclic subgroups lying in~the~center of~$G$. Suppose also that $G$ is residually a~$\mathcal{C}\mbox{-}\mathcal{BN}_{\mathfrak{P}(\mathcal{C})}$\nobreakdash-group and~has a~homomorphism onto~a~$\mathcal{C}\mbox{-}\mathcal{BN}_{\mathfrak{P}(\mathcal{C})}$\nobreakdash-group that acts injectively on~$HK$. Then $G^{*}$ is residually a~$\mathcal{C}$\nobreakdash-group if and~only~if

\textup{1)}\hspace{1ex}$H/H \cap K$ and~$K/H \cap K$ are of~the~same order;

\textup{2)}\hspace{1ex}$H$ and~$K$ are $\mathfrak{P}(\mathcal{C})^{\prime}$\nobreakdash-iso\-lat\-ed~in~$G$;

\textup{3)}\hspace{1ex}$2 \in \mathfrak{P}(\mathcal{C})$, unless $H \cap K$ lies in~the~center of~$G^{*}$.
\end{etheorem}

In~what follows, it is assumed that $\Gamma$ is a~non-empty undirected connected graph with~a~vertex set~$\mathcal{V}$ and~an~edge set~$\mathcal{E}$ (loops and~multiple edges are allowed). Let us construct an~oriented graph of~groups~$\mathcal{G}(\Gamma)$ over~$\Gamma$. To~do this, we first choose arbitrarily the~directions for~all edges of~$\Gamma$ and,~for~every edge $e \in \mathcal{E}$, denote by~$e(1)$ and~$e(-1)$ the~vertices that are the~ends of~$e$. Then we assign to~each vertex $v \in \mathcal{V}$ some group~$G_{v}$ and~to~each edge $e \in \mathcal{E}$ a~group~$H_{e}$ and~injective homomorphisms $\varphi_{+e}\colon H_{e} \to G_{e(1)}$, $\varphi_{-e}\colon H_{e} \to G_{e(-1)}$. The~result is the~graph of~groups
$$
\mathcal{G}(\Gamma) = \big(\Gamma,\ G_{v}\ (v \in \mathcal{V}),\ H_{e}\ (e \in \mathcal{E}),\ \varphi_{\varepsilon e}\ (e \in \mathcal{E},\ \varepsilon = \pm 1)\big).
$$

The~following two theorems deal with~the~fundamental group~$\pi_{1}(\mathcal{G}(\Gamma))$ of~$\mathcal{G}(\Gamma)$, the~definition of~which can be found in~\cite{Serre1980}. Recall that if the~graph~$\Gamma$ is a~tree, then $\pi_{1}(\mathcal{G}(\Gamma))$ is said to~be the~\emph{tree product} of~the~groups~$G_{v}$ ($v \in \mathcal{V}$)~\cite{KarrasSolitar1970TAMS}.

\begin{etheorem}\label{te305}
Suppose that $\Gamma$ is a~finite tree and~$H_{e}\varphi_{\varepsilon e}$ is a~proper central subgroup of~$G_{e(\varepsilon)}$ for~all $e \in \mathcal{E}$, $\varepsilon = \pm 1$. Suppose also that, for~each $v \in \mathcal{V}$, $G_{v}$~is residually a~$\mathcal{C}\mbox{-}\mathcal{BN}_{\mathfrak{P}(\mathcal{C})}$\nobreakdash-group and~has a~homomorphism onto~a~$\mathcal{C}\mbox{-}\mathcal{BN}_{\mathfrak{P}(\mathcal{C})}$\nobreakdash-group that acts injectively on~all the~subgroups~$H_{e}\varphi_{\varepsilon e}$ \textup{(}$e \in \mathcal{E}$, $\varepsilon = \pm 1$, $v = e(\varepsilon)$\textup{)}. Then $\pi_{1}(\mathcal{G}(\Gamma))$ is residually a~$\mathcal{C}$\nobreakdash-group if and~only if, for~any $e \in \mathcal{E}$, $\varepsilon = \pm 1$, $H_{e}\varphi_{\varepsilon e}$ is $\mathfrak{P}(\mathcal{C})^{\prime}$\nobreakdash-iso\-lat\-ed in~$G_{e(\varepsilon)}$.
\end{etheorem}

We say that $\mathcal{G}(\Gamma)$ is a~\emph{graph of~groups with~central trivially intersecting edge subgroups} if, for~each $v \in \mathcal{V}$, the~subgroup
$$
H_{v} = \operatorname{sgp}\big\{H_{e}\varphi_{\varepsilon e} \mid e \in \mathcal{E},\ \varepsilon = \pm 1,\ v = e(\varepsilon)\big\}
$$
lies in~the~center of~$G_{v}$ and~is the~direct product of~the~subgroups generating~it.

\begin{etheorem}\label{te306}
Suppose that $\mathcal{G}(\Gamma)$ is a~graph of~groups with~central trivially intersecting edge subgroups. Suppose also that, for~each $v \in \mathcal{V}$, $G_{v}$~is residually a~$\mathcal{C}\mbox{-}\mathcal{BN}_{\mathfrak{P}(\mathcal{C})}$\nobreakdash-group and~has a~homomorphism onto~a~$\mathcal{C}\mbox{-}\mathcal{BN}_{\mathfrak{P}(\mathcal{C})}$\nobreakdash-group that acts injectively on~$H_{v}$. Then the~following statements hold.

\textup{1.}\hspace{1ex}If, for~each $v \in \mathcal{V}$, $H_{v}$~is $\mathfrak{P}(\mathcal{C})^{\prime}$\nobreakdash-iso\-lat\-ed in~$G_{v}$, then $\pi_{1}(\mathcal{G}(\Gamma))$ is residually a~$\mathcal{C}$\nobreakdash-group.

\textup{2.}\hspace{1ex}Suppose that $\mathcal{G}(\Gamma)$ is finite and~$H_{e}\varphi_{\varepsilon e} \ne G_{e(\varepsilon)}$ for~all $e \in \mathcal{E}$, $\varepsilon = \pm 1$. Then $\pi_{1}(\mathcal{G}(\Gamma))$ is residually a~$\mathcal{C}$\nobreakdash-group if and~only if, for~any $e \in \mathcal{E}$, $\varepsilon = \pm 1$, $H_{e}\varphi_{\varepsilon e}$~is $\mathfrak{P}(\mathcal{C})^{\prime}$\nobreakdash-iso\-lat\-ed~in~$G_{e(\varepsilon)}$.
\end{etheorem}

The~remaining part of~the~article is organized as~follows. In~Sections~\ref{se05} and~\ref{se06}, we give equivalent definitions of~a~(weakly) $\mathcal{C}$\nobreakdash-bound\-ed abelian group, introduce the~notion of~a~(weakly) $\mathcal{C}$\nobreakdash-bound\-ed solvable group, and~establish a~number of~properties of~$\mathcal{C}$\nobreakdash-bound\-ed groups. In~Sections~\ref{se07} and~\ref{se08}, all the~theorems and~corollaries formulated above are proved. Section~\ref{se04} contains some auxiliary assertions used in~the~proofs. In~Section~\ref{se09}, we construct the~above-mentioned example of~a~weakly $\mathcal{C}$\nobreakdash-bound\-ed nilpotent group that does not have the~property~$\mathcal{C}\mbox{-}\mathfrak{Sep}$.

\section{Some auxiliary statements}\label{se04}

The~next three assertions describe some families of~groups that necessarily belong to~a~given root class.

\begin{eproposition}\label{pe401}
\textup{\cite[Proposition~8]{Tumanova2019SMJ}}
Let $\mathcal{C}$ be a~root class of~groups consisting only of~periodic groups. A~finite solvable group belongs to~$\mathcal{C}$ if and~only if its order is a~$\mathfrak{P}(\mathcal{C})$\nobreakdash-num\-ber. In~particular, $\mathcal{C}$ contains all finite $p$\nobreakdash-groups for~some prime number~$p$.
\end{eproposition}

\begin{eproposition}\label{pe402}
\textup{\cite[Proposition~17]{SokolovTumanova2020IVM}}
If $\mathcal{C}$ is a~root class of~groups consisting only of~periodic groups, then an~arbitrary $\mathcal{C}$\nobreakdash-group is of~finite exponent.
\end{eproposition}

\begin{eproposition}\label{pe403}
If $\mathcal{C}$ is a~root class of~groups, $X$ is a~non-trivial $\mathcal{C}$\nobreakdash-group, and~$\mathfrak{c}$ is the~cardinality of~$X$, then the~following statements hold.

\textup{1.}\hspace{1ex}If $\mathcal{C}$ contains at~least one non-periodic group, then it includes all free solvable groups whose cardinalities do not exceed~$2^{\mathfrak{c}}$. In~particular, $\mathcal{C}$ contains any free solvable group of~cardinality less than~$\aleph_{\infty}$.

\textup{2.}\hspace{1ex}If $\mathcal{C}$ consists only of~periodic groups, then it includes all periodic solvable $\mathfrak{P}(\mathcal{C})$\nobreakdash-groups of~finite exponent whose cardinalities do not exceed~$2^{\mathfrak{c}}$. In~particular, together with~some infinite group, $\mathcal{C}$~contains any periodic solvable $\mathfrak{P}(\mathcal{C})$\nobreakdash-group that has a~finite exponent and~a~cardinality less than~$\aleph_{\infty}$.
\end{eproposition}

\begin{proof}
Let us prove Statements~1 and~2 simultaneously.

Suppose that $Y^{(\alpha)}$ and~$Y^{(\beta)}$ are a~free solvable group and~a~periodic solvable $\mathfrak{P}(\mathcal{C})$\nobreakdash-group of~finite exponent, respectively, whose cardinalities do not exceed~$2^{\mathfrak{c}}$ (hereinafter, groups with~index~$\beta$ are considered only if $\mathcal{C}$ consists of~periodic groups and~therefore the~set~$\mathfrak{P}(\mathcal{C})$ is defined). Suppose also that $Z^{(\alpha)}$ is an~infinite cyclic group, $q$ is the~exponent of~$Y^{(\beta)}$ (which is a~$\mathfrak{P}(\mathcal{C})$\nobreakdash-num\-ber), and~$Z^{(\beta)}$ is a~cyclic group of~order~$q$. We note that, if $\mathcal{C}$ includes at~least one non-periodic group, then $Z^{(\alpha)} \in \mathcal{C}$ because $\mathcal{C}$ is closed under taking subgroups; if $\mathcal{C}$ consists only of~periodic groups, then $Z^{(\beta)} \in \mathcal{C}$ by~Proposition~\ref{pe401}.

Let us put
$$
\mathcal{I} = \prod_{x \in X}X_{x}^{\vphantom{(k)}},\quad
D^{(k)}_{\vphantom{i}} = \prod_{i \in \mathcal{I}}Z_{i}^{(k)},
$$
where $k \in \{\alpha, \beta\}$, $\prod$~stands for~Cartesian product, $X_{x}^{\vphantom{(k)}}$ and~$Z_{i}^{(k)}$ are isomorphic copies of~$X$ and~$Z^{(k)}_{\vphantom{i}}$ respectively for~all $x \in X$, $i \in \mathcal{I}$. Then $\mathcal{I}, D^{(\alpha)}_{\vphantom{i}}, D^{(\beta)}_{\vphantom{i}} \in \mathcal{C}$ by~the~definition of~a~root class, and~$2^{\mathfrak{c}} \leqslant \operatorname{card}\mathcal{I} \leqslant 2^{\mathfrak{c} \cdot \mathfrak{c}}$.

Let $F^{(k)}$ be an~arbitrary factor of~the~derived series of~$Y^{(k)}$, $k \in \{\alpha, \beta\}$. Then $F^{(\alpha)}$ is a~free abelian group and~$F^{(\beta)}$ is a~periodic abelian group whose exponent is finite and~divides~$q$. The~latter means that $F^{(\beta)}$ has a~finite number of~$p$\nobreakdash-power torsion subgroups. By~the~first Pr\"ufer theorem, each of~these subgroups can be decomposed into~a~direct product of~cyclic subgroups whose orders also divide~$q$. Therefore, it follows from~the~relations
$$
\operatorname{card}F^{(k)} \leqslant \operatorname{card}Y^{(k)} \leqslant 2^{\mathfrak{c}},\quad
k \in \{\alpha, \beta\},
$$
that $F^{(k)} \leqslant D^{(k)}$. Thus, $F^{(k)} \in \mathcal{C}$ and~$Y^{(k)} \in \mathcal{C}$ because $\mathcal{C}$ is closed under taking subgroups and~extensions.
\end{proof}

\begin{eproposition}\label{pe404}
\textup{\cite[Proposition~3]{SokolovTumanova2020IVM}}
If $\mathcal{C}$ is an~arbitrary class of~groups, $X$ is a~group, and~$Y$ is a~normal subgroup~of~$X$, then the~following statements hold.

\textup{1.}\hspace{1ex}If $X/Y$ is residually a~$\mathcal{C}$\nobreakdash-group, then $Y$ is $\mathcal{C}$\nobreakdash-sepa\-ra\-ble~in~$X$.

\textup{2.}\hspace{1ex}If $\mathcal{C}$ is closed under taking quotient groups and~$Y$ is $\mathcal{C}$\nobreakdash-sepa\-ra\-ble in~$X$, then $X/Y$ is residually a~$\mathcal{C}$\nobreakdash-group.
\end{eproposition}

\begin{eproposition}\label{pe405}
\textup{\cite[Proposition~5]{SokolovTumanova2016SMJ}}
If $\mathcal{C}$ is a~class of~groups consisting only of~periodic groups, $X$ is a~group, and~$Y$ is a~$\mathcal{C}$\nobreakdash-sepa\-ra\-ble subgroup of~$X$, then $Y$ is $\mathfrak{P}(\mathcal{C})^{\prime}$\nobreakdash-iso\-lat\-ed~in~$X$.
\end{eproposition}

In~what follows, if $\mathcal{C}$ is a~class of~groups and~$X$ is a~group, then $\mathcal{C}^{*}(X)$ denotes the~family of~normal subgroups of~$X$ such that $Y \in \mathcal{C}^{*}(X)$ if and~only if $X/Y \in \mathcal{C}$.

\begin{eproposition}\label{pe406}
\textup{\cite[Proposition~1]{SokolovTumanova2016SMJ}}
If $\mathcal{C}$ is a~class of~groups closed under taking subgroups and~finite direct products, then, for~any group~$X$, the~intersection of~a~finite number of~subgroups from~$\mathcal{C}^{*}(X)$ is again a~subgroup from~this family.
\end{eproposition}

Suppose that $\mathcal{C}$ is a~class of~groups, $X$ is a~group, and~$Y$ is a~subgroup of~$X$. It~is easy to~see that, in~$X$, the~intersection of~any number of~$\mathcal{C}$\nobreakdash-sepa\-ra\-ble subgroups is again a~$\mathcal{C}$\nobreakdash-sepa\-ra\-ble subgroup. Therefore, there exists the~smallest $\mathcal{C}$\nobreakdash-sepa\-ra\-ble subgroup of~$X$ containing~$Y$. We refer to~this subgroup as~the~\emph{$\mathcal{C}$\nobreakdash-clo\-sure} of~$Y$ in~$X$ and~denote it by~$\mathcal{C}\mbox{-}\mathfrak{Cl}(X, Y)$. It~follows from~Proposition~\ref{pe405} that if $\mathcal{C}$ consists only of~periodic groups, then, for~any $X$ and~$Y$, the~relation $\mathfrak{P}(\mathcal{C})^{\prime}\mbox{-}\mathfrak{Is}(X, Y) \leqslant \mathcal{C}\mbox{-}\mathfrak{Cl}(X, Y)$ holds.

\begin{eproposition}\label{pe407}
\textup{\cite[Proposition~3]{Sokolov2017SMJ}}
Given an~arbitrary class of~groups~$\mathcal{C}$, the~equality
$$
\mathcal{C}\mbox{-}\mathfrak{Cl}(X, Y) = \bigcap_{N \in \mathcal{C}^{*}(X)} YN
$$
holds for~any group~$X$ and~its subgroup~$Y$.
\end{eproposition}

\begin{eproposition}\label{pe408}
\textup{\cite[Proposition~4]{Sokolov2017SMJ}}
Suppose that $\mathcal{C}$ is an~arbitrary class of~groups, $X$ is residually a~$\mathcal{C}$\nobreakdash-group, and~$Y$ is a~subgroup of~$X$. If~$Y$ is a~nilpotent group of~class~$c$, then $\mathcal{C}\mbox{-}\mathfrak{Cl}(X, Y)$ is also a~nilpotent group of~class~$c$.
\end{eproposition}

\begin{eproposition}\label{pe409}
\textup{\cite[Theorems~5.3,~5.6\nobreakdash--5.8]{ClementMajewiczZyman2017}}
If $\mathfrak{P}$ is a~set of~primes, $X$ is a~locally nilpotent group, and~$Y$ is a~subgroup of~$X$, then the~following statements hold.

\textup{1.}\hspace{1ex}$\mathfrak{P}^{\prime}\mbox{-}\mathfrak{Rt}(X, Y) = \mathfrak{P}^{\prime}\mbox{-}\mathfrak{Is}(X, Y)$.

\textup{2.}\hspace{1ex}If $Y$ is $\mathfrak{P}^{\prime}$\nobreakdash-iso\-lat\-ed in~$X$, then the~normalizer of~$Y$ in~$X$ is also $\mathfrak{P}^{\prime}$\nobreakdash-iso\-lat\-ed~in~$X$.

\textup{3.}\hspace{1ex}If $X$ is $\mathfrak{P}^{\prime}$\nobreakdash-tor\-sion-free, then all the~members of~its upper central series are $\mathfrak{P}^{\prime}$\nobreakdash-iso\-lat\-ed in~$X$. The~extraction of~$\mathfrak{P}^{\prime}$\nobreakdash-roots is unique in~such a~group.
\end{eproposition}

\begin{eproposition}\label{pe410}
\textup{\cite[Lemma~2]{Malcev1958PIvPI}}
If $X$ is a~nilpotent group of~class~$c$, then, for~each $y \in X^{n^{c}}$, the~equation $x^{n} = y$ is solvable~in~$X$.
\end{eproposition}

\section{The~classes of~$\mathcal{C}$-bounded abelian, nilpotent, and~solvable groups}\label{se05}

Let $\mathcal{C}$ be a~class of~groups consisting only of~periodic groups, and~let $A$ be an~abelian group. Consider the~following set of~conditions:

$(1)_{\phantom{\mathcal{C}}}$\hspace{1ex}$A$ is of~finite rank;

$(2)_{\mathcal{C}}$\hspace{1ex}an~arbitrary quotient group of~$A$ does not contain $p$\nobreakdash-quasi\-cyclic subgroups for~any $p \in \mathfrak{P}(\mathcal{C})$;

$(3)_{\mathcal{C}}$\hspace{1ex}$A \in \mathcal{C}\mbox{-}w\mathcal{BA}$;

$(4)_{\mathcal{C}}$\hspace{1ex}each primary $\mathfrak{P}(\mathcal{C})$\nobreakdash-com\-po\-nent of~$A$ has a~finite exponent and~a~cardinality not~exceeding the~cardinality of~some $\mathcal{C}$\nobreakdash-group;

$(5)_{\mathcal{C}}$\hspace{1ex}$A \in \mathcal{C}\mbox{-}\mathcal{BA}$.

\begin{eproposition}\label{pe501}
If $\mathcal{C}$ is a~root class of~groups consisting only of~periodic groups and~$A$ is an~abelian group, then the~following statements hold.

\textup{1.}\hspace{1ex}$(2)_{\mathcal{C}} \Rightarrow (1)$.

\textup{2.}\hspace{1ex}$(2)_{\mathcal{C}} \Leftrightarrow (3)_{\mathcal{C}}$.

\textup{3.}\hspace{1ex}$(3)_{\mathcal{C}} \wedge (4)_{\mathcal{C}} \Leftrightarrow (5)_{\mathcal{C}}$.
\end{eproposition}

\begin{proof}
1.\hspace{1ex}Assume that the~rank of~$A$ is infinite. Then $A$ contains a~free abelian subgroup~$B$ of~infinite rank, and~this subgroup can be mapped homomorphically onto~a~$p$\nobreakdash-quasi\-cyclic group for~any prime number~$p$. Because $A$ is abelian, any homomorphism of~$B$ can be extended to~a~homomorphism of~$A$. Since $\mathcal{C}$ contains non-trivial groups, $\mathfrak{P}(\mathcal{C}) \ne \varnothing$. Therefore, $A$ does not satisfy~$(2)_{\mathcal{C}}$.
\smallskip

2.\hspace{1ex}The~sufficiency of~the~statement is obvious; let us verify the~necessity.

Assume that a~quotient group~$B$ of~$A$ has elements $b_{1},\, b_{2},\, b_{3},\, \ldots$ whose orders equal $p,\, p^{2}\kern-2pt{},\, p^{3}\kern-2pt{},\, \ldots$ respectively for~some $p \in \mathfrak{P}(\mathcal{C})$. The~group $N = \operatorname{sgp}\{b_{1}, b_{2}, b_{3}, \ldots\}$ is countable and,~by~the~second Pr\"ufer theorem, either decomposes into~a~direct product of~cyclic subgroups, the~orders of~which can be arbitrarily large, or~contains a~non-trivial element~$b$ of~infinite height. In~the~first case, $N$ can be mapped homomorphically onto~a~$p$\nobreakdash-quasi\-cyclic group, and~hence there exists a~quotient group of~$A$ containing such a~subgroup. Let us show that, in~the~second case, there is also a~homomorphism of~$A$ onto~a~$p$\nobreakdash-quasi\-cyclic group and~thus $A$ does not satisfy~$(2)_{\mathcal{C}}$.

It is well known (see, e.g.,~\cite[Sec.~4.1]{Robinson1996}) that $B$ can be embedded into~a~divisible abelian group~$D$ and~that the~latter can be decomposed into~a~direct product of~quasicyclic groups and~groups isomorphic to~the~additive group of~rational numbers. Since $N$ is a~periodic $p$\nobreakdash-group, this decomposition has $p$\nobreakdash-quasi\-cyclic factors and~there exists a~homomorphism~$\sigma$ of~$D$ onto~one of~these factors mapping~$b$ to~a~non-trivial element. Since the~height of~$b$ is infinite, the~restriction of~$\sigma$ to~$B$ cannot have a~finite image. Therefore, the~composition of~the~natural homomorphism $A \to B$ and~the~specified restriction is the~desired mapping.\smallskip

3.\hspace{1ex}As~above, only the~necessity has to~be proved.

Let $M$ be a~subgroup of~$A$. It~follows from~Statements~1 and~2 that $M$ is of~finite rank. Denote by~$N$ the~subgroup generated by~some maximal linearly independent subset of~$M$. Since all the~elements of~$M/N$ have finite orders, the~torsion subgroup~$\tau(A/M)$ of~$A/M$ is a~homomorphic image of~the~torsion subgroup~$\tau(A/N)$ of~$A/N$. Because $\tau(A/M)$ decomposes into~the~direct product of~its $p$\nobreakdash-power torsion subgroups, each of~these subgroups is also a~homomorphic image of~$\tau(A/N)$ and~hence is a~quotient group of~the~corresponding $p$\nobreakdash-power torsion subgroup of~$A/N$. Therefore, it~suffices to~show that any primary $\mathfrak{P}(\mathcal{C})$\nobreakdash-com\-po\-nent of~$A/N$ has a~cardinality not~exceeding the~cardinality of~some $\mathcal{C}$\nobreakdash-group.

So, let $T/N = \tau_{p}(A/N)$ be a~primary $\mathfrak{P}(\mathcal{C})$\nobreakdash-com\-po\-nent of~$A/N$. Denote by~$q$ the~exponent of~this group, which is finite by~the~condition~$(3)_{\mathcal{C}}$. Let also $\sigma$ be the~endomorphism of~$T$ that raises each of~its elements to~the~power~$q$. Then $T\sigma \leqslant N$ and~the~kernel~$S$ of~$\sigma$ is contained in~the~corresponding primary $\mathfrak{P}(\mathcal{C})$\nobreakdash-com\-po\-nent~$\tau_{p}(A)$ of~$A$. By~the~condition~$(4)_{\mathcal{C}}$, the~cardinality of~$\tau_{p}(A)$ does not exceed the~cardinality~$\mathfrak{c}$ of~some $\mathcal{C}$\nobreakdash-group. As~for~$N$, it is finitely generated and~therefore at~most countable. Hence, if $S$ is infinite, then the~cardinalities of~$T$ and~$T/N$ do not exceed~$\mathfrak{c}$. Otherwise, $T$ is an~extension of~a~finite group by~a~finitely generated one. Therefore, it is finitely generated, the~quotient group~$T/N$ is finite and,~by~Proposition~\ref{pe401}, has an~order that does not exceed the~order of~some $\mathcal{C}$\nobreakdash-group.
\end{proof}

It turns out that, in~some propositions given below and~concerning $\mathcal{C}$\nobreakdash-bound\-ed groups, nilpotency is too strong a~restriction and~can be replaced by~solvability. Therefore, we supplement the~list of~introduced concepts and~call a~solvable group (\emph{weakly}) \emph{$\mathcal{C}$\nobreakdash-bound\-ed} if it has at~least one finite subnormal series with~(weakly) $\mathcal{C}$\nobreakdash-bound\-ed abelian factors. Denote the~classes of~$\mathcal{C}$\nobreakdash-bound\-ed and~weakly $\mathcal{C}$\nobreakdash-bound\-ed solvable groups by~$\mathcal{C}\mbox{-}\mathcal{BS}$ and~$\mathcal{C}\mbox{-}w\mathcal{BS}$ respectively.

\begin{eproposition}\label{pe502}
If $\mathcal{C}$ is a~root class of~groups consisting only of~periodic groups, then the~following statements hold.\medskip\pagebreak

\textup{1.}\hspace{1ex}The~classes $\mathcal{C}\mbox{-}\mathcal{BA}$, $\mathcal{C}\mbox{-}w\mathcal{BA}$, $\mathcal{C}\mbox{-}\mathcal{BN}$, $\mathcal{C}\mbox{-}w\mathcal{BN}$, $\mathcal{C}\mbox{-}\mathcal{BS}$, and~$\mathcal{C}\mbox{-}w\mathcal{BS}$ are closed under taking subgroups, quotient groups, and~finite direct products.

\textup{2.}\hspace{1ex}Let $X$ be an~abelian group. If~$X\kern-1pt{} \in \mathcal{C}\mbox{-}\mathcal{BS}$ \textup{(}$X\kern-1pt{} \in \mathcal{C}\mbox{-}w\mathcal{BS}$\textup{)}, then $X\kern-1pt{} \in \mathcal{C}\mbox{-}\mathcal{BA}$ \textup{(}respectively $X\kern-1pt{} \in \mathcal{C}\mbox{-}w\mathcal{BA}$\textup{)}.
\end{eproposition}

\begin{proof}
1.\hspace{1ex}If $A$ is an~abelian group and~$B$ is a~subgroup of~$A$, then every homomorphic image of~the~quotient group~$A/B$ is simultaneously a~homomorphic image of~$A$, and~every homomorphic image of~$B$ embeds into~some homomorphic image of~$A$. Hence, if $A$ belongs to~the~class $\mathcal{C}\mbox{-}\mathcal{BA}$ or~$\mathcal{C}\mbox{-}w\mathcal{BA}$, then the~same class contains $B$ and~$A/B$.

Suppose that $U$ and~$V$ are $\mathcal{C}\mbox{-}w\mathcal{BA}$\nobreakdash-groups and~$P$ is their direct product. If~$Q$ is a~subgroup of~$P$, then $P/Q$ is an~extension of~$UQ/Q \cong U/U \cap Q$~by
$$
(P/Q)/(UQ/Q) \cong P/UQ = UV/UW \cong V/W(V \cap U) = V/W,
$$
where $W$ is the~image of~$Q$ under the~canonical projection of~$P$ onto~$V$. It~follows from~the~inclusions $U, V \in \mathcal{C}\mbox{-}w\mathcal{BA}$ that all the~primary $\mathfrak{P}(\mathcal{C})$\nobreakdash-com\-po\-nents of~$U/U \cap Q$ and~$V/W$ are of~finite exponent. Hence, the~primary $\mathfrak{P}(\mathcal{C})$\nobreakdash-com\-po\-nents of~$P/Q$ have the~same property and~therefore $P \in \mathcal{C}\mbox{-}w\mathcal{BA}$. Each primary $\mathfrak{P}(\mathcal{C})$\nobreakdash-com\-po\-nent of~$P$ is the~direct product of~the~corresponding primary $\mathfrak{P}(\mathcal{C})$\nobreakdash-com\-po\-nents of~$U$ and~$V$. Therefore, if $U, V \in \mathcal{C}\mbox{-}\mathcal{BA}$, then $P$ satisfies~$(4)_{\mathcal{C}}$ because $\mathcal{C}$ is closed under taking finite direct products, and~$P \in \mathcal{C}\mbox{-}\mathcal{BA}$ by~Proposition~\ref{pe501}.

Thus, Statement~1 holds for~the~classes $\mathcal{C}\mbox{-}w\mathcal{BA}$ and~$\mathcal{C}\mbox{-}\mathcal{BA}$. Suppose now that $X$ and~$Y$ are nilpotent (solvable) groups,
$$
1 = X_{0} \leqslant X_{1} \leqslant \ldots \leqslant X_{m} = X
\quad\text{and}\quad
1 = Y_{0} \leqslant Y_{1} \leqslant \ldots \leqslant Y_{n} = Y
$$
are their central (subnormal) series, and~$Z$ is a~subgroup of~$X$. Without loss of~generality, we can assume that $m = n$. Then the~subgroups $X_{i} \times Y_{i}$, $X_{i} \cap Z$, and~(if $Z$ is normal in~$X)$ $X_{i}Z/Z$ ($0 \leqslant i \leqslant n$) form central (subnormal) series of~$X \times Y$, $Z$, and~$X/Z$. The~factors of~these series are isomorphic respectively to~the~direct products, subgroups, and~homomorphic images of~$X_{i+1}/X_{i}$ and~$Y_{i+1}/Y_{i}$ ($0 \leqslant i \leqslant n-1$). Therefore, it follows from~the~properties of~the~classes $\mathcal{C}\mbox{-}\mathcal{BA}$ and~$\mathcal{C}\mbox{-}w\mathcal{BA}$ proved above that Statement~1 also holds for~the~classes $\mathcal{C}\mbox{-}\mathcal{BN}$, $\mathcal{C}\mbox{-}\mathcal{BS}$, $\mathcal{C}\mbox{-}w\mathcal{BN}$, and~$\mathcal{C}\mbox{-}w\mathcal{BS}$.
\smallskip

2.\hspace{1ex}Let $Y$ be a~quotient group of~$X$. By~Statement~1, each primary $\mathfrak{P}(\mathcal{C})$\nobreakdash-com\-po\-nent~$T$ of~$Y$ is a~$\mathcal{C}\mbox{-}w\mathcal{BS}$\nobreakdash-group, i.e.,~has a~subnormal series with~$\mathcal{C}\mbox{-}w\mathcal{BA}$\nobreakdash-factors. Any factor~$F$ of~this series is a~$p$\nobreakdash-group for~some $p \in \mathfrak{P}(\mathcal{C})$ and~hence is of~finite exponent. Therefore, the~exponent of~$T$ is also finite and~$X \in \mathcal{C}\mbox{-}w\mathcal{BA}$. If~$X \in \mathcal{C}\mbox{-}\mathcal{BS}$, then $T \in \mathcal{C}\mbox{-}\mathcal{BS}$, $F \in \mathcal{C}\mbox{-}\mathcal{BA}$, and~the~cardinality of~$F$ does not exceed the~cardinality of~some $\mathcal{C}$\nobreakdash-group. The~group~$T$ also enjoys the~last property because $\mathcal{C}$ is closed under taking extensions. Therefore, $X \in \mathcal{C}\mbox{-}\mathcal{BA}$.
\end{proof}

It follows from~Proposition~\ref{pe502}, in~particular, that a~nilpotent (solvable) group is (weakly) $\mathcal{C}$\nobreakdash-bound\-ed if and~only if the~factors of~\emph{any} of~its central (solvable) series are (weakly) $\mathcal{C}$\nobreakdash-bound\-ed abelian groups.

\section{$\mathcal{C}$-regularity and~$\mathcal{C}$-quasiregularity}\label{se06}

Suppose that $\mathcal{C}$ is a~class of~groups and~$X$ is a~group. Let us say that~$X$

--\hspace{1ex}is \emph{$\mathcal{C}$\nobreakdash-quasi\-regular with~respect to~its subgroup~$Y$} if, for~any subgroup $M \in \mathcal{C}^{*}(Y)$, there exists a~subgroup $N \in \mathcal{C}^{*}(X)$ such that $N \cap Y \leqslant M$;

--\hspace{1ex}is \emph{$\mathcal{C}$\nobreakdash-regular with~respect to~its \emph{normal} subgroup~$Y$} if, for~any subgroup $M \in \mathcal{C}^{*}(Y)$ \emph{normal in~$X$}, there exists a~subgroup $N \in \mathcal{C}^{*}(X)$ such that $N \cap Y = M$.

The~introduced concepts often serve as~parts of~conditions sufficient for~a~free construction of~group to~be residually a~$\mathcal{C}$\nobreakdash-group. The~main goal of~this section is to~prove that a~$\mathcal{C}\mbox{-}\mathcal{BN}$\nobreakdash-group has the~properties of~$\mathcal{C}$\nobreakdash-regularity and~$\mathcal{C}$\nobreakdash-quasi\-regularity for~any root class~$\mathcal{C}$ consisting only of~periodic groups.

\begin{eproposition}\label{pe601}
Let $\mathcal{C}$ be a~root class of~groups consisting only of~periodic groups. If~a~periodic $\mathcal{C}\mbox{-}\mathcal{BS}$\nobreakdash-group~$X$ has a~finite exponent, which is a~$\mathfrak{P}(\mathcal{C})$\nobreakdash-num\-ber, then $X \in \mathcal{C}$.
\end{eproposition}

\begin{proof}
By~Proposition~\ref{pe502}, an~arbitrary factor~$F$ of~some solvable series of~$X$ belongs to~the~class $\mathcal{C}\mbox{-}\mathcal{BA}$. Since $F$ has a~finite exponent, which is a~$\mathfrak{P}(\mathcal{C})$\nobreakdash-num\-ber, then the~number of~its $p$\nobreakdash-power torsion subgroups is finite, they all correspond to~numbers from~$\mathfrak{P}(\mathcal{C})$ and~therefore have cardinalities not~exceeding the~cardinality of~some $\mathcal{C}$\nobreakdash-group. Hence, the~cardinality of~$X$ also does not exceed the~cardinality of~some $\mathcal{C}$\nobreakdash-group because $\mathcal{C}$ is closed under taking extensions, and~$X \in \mathcal{C}$ by~Proposition~\ref{pe403}.
\end{proof}

\begin{eproposition}\label{pe602}
Suppose that $\mathcal{C}$ is a~root class of~groups consisting only of~periodic groups, $X$ is a~$\mathcal{C}\mbox{-}w\mathcal{BN}$\nobreakdash-group, and~$Y$ is a~periodic subgroup of~$X$. If~the~exponent~$m$ of~$Y$ is finite and~is a~$\mathfrak{P}(\mathcal{C})$\nobreakdash-num\-ber, then $Y \cap X^{m^{n}} = 1$ for~some $n \geqslant 1$.
\end{eproposition}

\begin{proof}
Suppose that $\mathfrak{S}$ is the~set of~all the~prime divisors of~$m$~and
$$
1 = X_{0} \leqslant X_{1} \leqslant \ldots \leqslant X_{n} = X
$$
is some central series of~$X$\kern-1pt{}.\kern-1pt{} Since $X\kern-3pt{} \in\kern-1pt{} \mathcal{C}\mbox{-}w\mathcal{BN}$\kern-1pt{} and~$\mathfrak{S}\kern-1pt{} \subseteq\kern-1pt{} \mathfrak{P}(\mathcal{C})$,\kern-1pt{} then,\kern-1pt{} for~any $i\kern-1pt{} \in\nolinebreak\kern-1pt{} \{0, \ldots, n-\nolinebreak1\}$, $p \in \mathfrak{S}$, the~relation $X_{i+1}/X_{i} \in \mathcal{C}\mbox{-}w\mathcal{BA}$ holds and~the~exponent~$e_{i,p}$ of~the~$p$\nobreakdash-power torsion subgroup of~$X_{i+1}/X_{i}$ is finite. Because $\mathfrak{S}$ is also finite, the~product
$$
q = \prod_{\substack{0 \leqslant i \leqslant n-1,\\p \in \mathfrak{S}}} e_{i,p}
$$
is defined.

It is easy to~see that, if $x \in X$ and~the~order of~$x$ is an~$\mathfrak{S}$\nobreakdash-num\-ber, then $x^{q} = 1$. Therefore, the~equation $x^{q} = y$ is not solvable in~$X$ for~any $y \in Y \setminus \{1\}$. But Proposition~\ref{pe410} says that it is solvable in~$X$ for~every $y \in X^{q^{c}}$, where $c$ is the~nilpotency class of~$X$. Hence, $Y \cap X^{q^{c}} = 1$. Since any prime divisor of~$q$ also divides~$m$, then $q|m^{d}$ for~some $d \geqslant 1$ and~the~number $n = cd$ is the~desired~one.
\end{proof}

\begin{eproposition}\label{pe603}
Suppose that $\mathcal{C}$ is a~root class of~groups consisting only of~periodic groups and~$\mathcal{NC}$ is the~class of~all nilpotent $\mathcal{C}$\nobreakdash-groups. Suppose also that $X$ is a~group, $Y$ is a~subgroup of~$X$, and~there exists a~homomorphism~$\sigma$ of~$X$ onto~a~$\mathcal{C}\mbox{-}\mathcal{BN}$\nobreakdash-group that acts injectively on~$Y$. Then the~following statements hold.

\textup{1.}\hspace{1ex}For~any subgroup $M \in \mathcal{C}^{*}(Y)$, there exists a~subgroup $N \in \mathcal{NC}^{*}(X)$ such that $N \cap\nolinebreak Y \leqslant\nolinebreak M$. In~particular, $X$ is $\mathcal{C}$\nobreakdash-quasiregular with~respect~to~$Y$.

\textup{2.}\hspace{1ex}Suppose that $Y$ is normal in~$X$. Then, for~any subgroup $M \in \mathcal{C}^{*}(Y)$ normal in~$X$, there exists a~subgroup $N \in \mathcal{NC}^{*}(X)$ such that $N \cap Y = M$. In~particular, $X$ is $\mathcal{C}$\nobreakdash-regular with~respect~to~$Y$.
\end{eproposition}

\begin{proof}
1.\hspace{1ex}We fix some subgroup $M \in \mathcal{C}^{*}(Y)$ and~assume first that $X \in \mathcal{C}\mbox{-}\mathcal{BN}$. Denote the~normalizer~$N_{X}(Y)$ of~$Y$ in~$X$ by~$Y_{1}$ and~inductively the~subgroup~$N_{X}(Y_{i})$ by~$Y_{i+1}$. It~is well known (see, e.g.,~\cite[Theorem~2.9]{ClementMajewiczZyman2017}) that $Y_{n} = X$ for~some~$n$. Let us argue by~induction~on~$n$.

Suppose that $n = 1$, i.e.,~$Y$ is normal in~$X$, and~denote by~$q$ the~exponent of~the~$\mathcal{C}$\nobreakdash-group $Y/M$, which is finite by~Proposition~\ref{pe402}. Then the~subgroup~$Y^{q}$ is also normal in~$X$ and~$Y^{q} \leqslant M$. The~quotient group $\overline{X} = X/Y^{q}$ belongs to~the~class $\mathcal{C}\mbox{-}\mathcal{BN}$ by~Proposition~\ref{pe502}, and~the~exponent of~the~subgroup $\overline{Y} = Y/Y^{q}$ is a~$\mathfrak{P}(\mathcal{C})$\nobreakdash-num\-ber. Therefore, it follows from~Propositions~\ref{pe601} and~\ref{pe602} \pagebreak that $\overline{Y} \cap \overline{X}^{\,r} = 1$ for~some $\mathfrak{P}(\mathcal{C})$\nobreakdash-num\-ber~$r$ and~$\overline{X}/\overline{X}^{\,r} \in \mathcal{C}$. Let $N$ be the~preimage of~$\overline{X}^{\,r}$ under the~natural homomorphism $X \to \overline{X}$. Then $N \in \mathcal{C}^{*}(X)$ and~$N \cap Y \leqslant Y^{q} \leqslant M$. Thus, $N$ is the~required subgroup.

Suppose now that $n \geqslant 2$. Since $Y$ is normal in~$Y_{1}$, then, as~proved above, there exists a~subgroup $N_{1} \in \mathcal{C}^{*}(Y_{1})$ satisfying the~condition $N_{1} \cap Y \leqslant M$. We apply the~inductive hypothesis to~the~subgroups~$Y_{1}$,~$N_{1}$ and~find a~subgroup $N \in \mathcal{C}^{*}(X)$ such that $N \cap Y_{1} \leqslant N_{1}$. Then $$N \cap Y = N \cap Y_{1} \cap Y \leqslant N_{1} \cap Y \leqslant M$$ and~therefore $N$ is the~desired subgroup.

In~the~general case (i.e.,~if $X$ does not necessarily belong to~$\mathcal{C}\mbox{-}\mathcal{BN}$), it follows from the~above and~the~relations
$$
X\sigma \in \mathcal{C}\mbox{-}\mathcal{BN},\quad
Y\sigma/M\sigma \cong Y/M(Y \cap \ker\sigma) = Y/M \in \mathcal{C}
$$
that there exists a~subgroup $\overline{N} \in \mathcal{C}^{*}(X\sigma)$ satisfying the~condition $\overline{N} \cap Y\sigma \leqslant M\sigma$. Since $X\sigma$ is nilpotent, then $\overline{N} \in \mathcal{NC}^{*}(X\sigma)$. Therefore, the~preimage~$N$ of~$\overline{N}$ under $\sigma$ belongs to~$\mathcal{NC}^{*}(X)$, and~$N \cap Y \leqslant M$ because $Y \cap \ker\sigma = 1$. Hence, $N$ is the~required subgroup.{\parfillskip=0pt\par}
\smallskip

2.\hspace{1ex}If a~subgroup $M \in \mathcal{C}^{*}(Y)$ is normal in~$X$, then $M\sigma$ is normal in~$X\sigma$ and~$X\sigma/M\sigma \in \mathcal{C}\mbox{-}\mathcal{BN}$ by~Proposition~\ref{pe502}. As proved above, $Y\sigma/M\sigma \in \mathcal{C}$. Therefore, we can apply Statement~1 to~the~group $\overline{X} = X\sigma/M\sigma$ and~the~subgroups $\overline{Y} = Y\sigma/M\sigma$ and~$\{1\}$. It~follows that there exists a~subgroup $\overline{N} \in \mathcal{NC}^{*}(\overline{X})$ such that $\overline{N} \cap \overline{Y} = 1$. Denote by~$N$ the~preimage of~$\overline{N}$ under the~composition of~$\sigma$ and~the~natural homomorphism $X\sigma \to \overline{X}$. It~is easy to~see that $N \in \mathcal{NC}^{*}(X)$ and~$N \cap Y = M$. Thus, $N$ is the~desired subgroup.
\end{proof}

\section{Proofs of~Theorem~\ref{te201}--\ref{te204} and~Corollary~\ref{ce203}}\label{se07}

\begin{proof}[\textup{\textbf{Proof of~Theorem~\ref{te201}}}]
1.\hspace{1ex}Assume that $X \notin \mathcal{C}\mbox{-}w\mathcal{BA}$. Then, by~Proposition~\ref{pe501}, $X$~does not satisfy~$(2)_{\mathcal{C}}$ and~hence there exists a~subgroup~$Y$\kern-1.5pt{} of~$X$ such that the~quotient group~$X/Y$ contains a~$p$\nobreakdash-quasi\-cyclic subgroup for~some $p \in \mathfrak{P}(\mathcal{C})$.

Let $T/Y = \{p\}^{\prime}\mbox{-}\mathfrak{Is}(X/Y, 1)$. Since $\{p\}^{\prime}\mbox{-}\mathfrak{Is}(X/Y, 1) = \{p\}^{\prime}\mbox{-}\mathfrak{Rt}(X/Y, 1)$ by~Proposition~\ref{pe409}, then the~quotient group $X/T \cong (X/Y)/(T/Y)$ also contains a~$p$\nobreakdash-quasi\-cyclic subgroup and~therefore is not residually a~$\mathcal{D}$\nobreakdash-group, where $\mathcal{D}$ is the~class of~all periodic $\mathfrak{P}(\mathcal{C})$\nobreakdash-groups of~finite exponent. Because this class is closed under taking quotient groups, it follows from~Proposition~\ref{pe404} that $T$ is not $\mathcal{D}$\nobreakdash-sepa\-ra\-ble in~$X$. Since $\mathcal{C} \subseteq \mathcal{D}$ by~Proposition~\ref{pe402}, then $T$ cannot be $\mathcal{C}$\nobreakdash-sepa\-ra\-ble in~$X$. At~the~same time, the~torsion subgroup of~$X/T$ is a~$p$\nobreakdash-group, so $T$ is $\mathfrak{P}(\mathcal{C})^{\prime}$\nobreakdash-iso\-lat\-ed~in~$X$.
\smallskip

2.\hspace{1ex}Let us choose a~$\mathfrak{P}(\mathcal{C})^{\prime}$\nobreakdash-iso\-lat\-ed subgroup~$Y$ of~$X$ and~show that the~quotient group $A = X/Y$ is residually a~$\mathcal{P}$\nobreakdash-group. Then it will follow from~Proposition~\ref{pe404} that $Y$ is $\mathcal{P}$\nobreakdash-sepa\-ra\-ble in~$X$. To~do this we fix a~non-trivial element~$a$ of~$A$ and~assume first that it belongs to~a~$p$\nobreakdash-power torsion subgroup~$T$ of~$A$ for~some prime number~$p$.

Since $Y$ is $\mathfrak{P}(\mathcal{C})^{\prime}$\nobreakdash-iso\-lat\-ed in~$X$, then $p \in \mathfrak{P}(\mathcal{C})$ and,~by~the~condition~$(3)_{\mathcal{C}}$, the~exponent~$q$ of~$T$ is finite. It is easy to~see that $A^{q} \cap T = 1$ and~hence $aA^{q} \ne 1$. It~remains to~note that, by~the~first Pr\"ufer theorem, the~periodic $p$\nobreakdash-group~$A/A^{q}$ decomposes into~a~direct product of~cyclic subgroups and~therefore is residually $p$\nobreakdash-finite.

If $a$ has a~finite order~$r$ that is not a~power of~a~prime number, the~proof reduces to~the~case considered above. It~is only necessary to~choose some prime divisor~$p$ of~$r$ and~pass to~the~quotient group~$A/\{p\}^{\prime}\mbox{-}\mathfrak{Is}(A, 1)$.

Suppose now that the~order of~$a$ is infinite. By~Proposition~\ref{pe502}, the~quotient group~$A/\langle a\rangle$ also satisfies~$(3)_{\mathcal{C}}$, and~therefore each of~its primary $\mathfrak{P}(\mathcal{C})$\nobreakdash-com\-po\-nents has a~finite exponent. It~follows that, for~any $p \in \mathfrak{P}(\mathcal{C})$, there exists a~number~$s$ such that $a \notin A^{p^{s}}$. Hence, we can use the~same reasoning as~in~the~first case, it~only remains to~note that $\mathfrak{P}(\mathcal{C}) \ne \varnothing$ (because $\mathcal{C}$ is a~root class), and~choose an~arbitrary number $p \in \mathfrak{P}(\mathcal{C})$.

Thus, any $\mathfrak{P}(\mathcal{C})^{\prime}$\nobreakdash-iso\-lat\-ed subgroup of~$X$ is $\mathcal{P}$\nobreakdash-sepa\-ra\-ble. Since $\mathfrak{P}(\mathcal{P}) = \mathfrak{P}(\mathcal{C})$ and~$\mathcal{P} \subseteq \mathcal{C}$ by~Proposition~\ref{pe401}, it follows that $X$ has the~properties $\mathcal{P}\mbox{-}\mathfrak{Sep}$ and~$\mathcal{C}\mbox{-}\mathfrak{Sep}$.
\end{proof}

The~next proposition generalizes Theorem~2.1 from~\cite{Gruenberg1957PLMS}.

\begin{eproposition}\label{pe701}
Let $\mathcal{C}$ be a~root class of~groups consisting only of~periodic groups. A~$\mathcal{C}\mbox{-}\mathcal{BN}$\nobreakdash-group~$X$ is residually a~$\mathcal{C}$\nobreakdash-group if and~only if it has no~$\mathfrak{P}(\mathcal{C})^{\prime}$\nobreakdash-tor\-sion.
\end{eproposition}

\begin{proof}
The~necessity follows from~Proposition~\ref{pe405}. To~prove the~sufficiency, we use induction on~the~nilpotency class of~$X$ and~assume that the~base of~the~induction is the~case when $X = 1$. Let us fix an~element $x \in X \setminus \{1\}$ and~indicate a~subgroup $N \in \mathcal{C}^{*}(X)$ such that $x \notin N$.

If $\mathcal{Z}(X)$ denotes the~center of~$X$, then, by~the~induction hypothesis and~Proposition~\ref{pe409}, $X/\mathcal{Z}(X)$ is residually a~$\mathcal{C}$\nobreakdash-group. Therefore, we can assume that $x \in \mathcal{Z}(X)$. Since $\mathcal{Z}(X) \in \mathcal{C}\mbox{-}\mathcal{BA}$ by~Proposition~\ref{pe502} and~the~trivial subgroup of~$\mathcal{Z}(X)$ is $\mathfrak{P}(\mathcal{C})^{\prime}$\nobreakdash-iso\-lat\-ed, then it follows from~Theorem~\ref{te201} that there exists a~subgroup $Y \in \mathcal{C}^{*}(\mathcal{Z}(X))$ not~containing~$x$. Therefore, the~element $\overline{x} = xY$ of~the~group $\overline{X} = X/Y$ is non-trivial and~has a~finite order, which is a~$\mathfrak{P}(\mathcal{C})$\nobreakdash-num\-ber. By~Propositions~\ref{pe502} and~\ref{pe401}, $\overline{X}$ is again a~$\mathcal{C}\mbox{-}\mathcal{BN}$\nobreakdash-group and~the~cyclic subgroup~$U$ generated by~$\overline{x}$ belongs to~$\mathcal{C}$. Hence, we can apply Proposition~\ref{pe603} to~the~group~$\overline{X}$ and~the~subgroups~$U$,~$\{1\}$, and~find a~subgroup $\overline{N} \in \mathcal{C}^{*}(\overline{X})$ such that $\overline{N} \cap U = 1$. Then $\overline{x} \notin \overline{N}$ and~therefore the~preimage~$N$ of~$\overline{N}$ under the~natural homomorphism $X \to \overline{X}$ is the~required subgroup.
\end{proof}

\begin{proof}[\textup{\textbf{Proof of~Theorem~\ref{te202}}}]
Let us denote for~convenience the~subgroup $\mathfrak{P}(\mathcal{C})^{\prime}\mbox{-}\mathfrak{Is}(X, Y)$ by~$\mathfrak{I}$. The~equality $\mathfrak{I} = \mathfrak{P}(\mathcal{C})^{\prime}\mbox{-}\mathfrak{Rt}(X, Y)$ follows from~Proposition~\ref{pe409}. If~$X$ has no~$\mathfrak{P}(\mathcal{C})^{\prime}$-tor\-sion, then it is residually a~$\mathcal{C}$\nobreakdash-group by~Proposition~\ref{pe701} and~the~nilpotency classes of~$\mathcal{C}\mbox{-}\mathfrak{Cl}(X, Y)$ and~$Y$ coincide by~Proposition~\ref{pe408}. It~follows from~Proposition~\ref{pe405} that $Y \leqslant \mathfrak{I} \leqslant \mathcal{C}\mbox{-}\mathfrak{Cl}(X, Y)$. Therefore, $\mathfrak{I}$ has the~same nilpotency class as~$Y$. Thus, it remains to~verify the~$\mathcal{C}$\nobreakdash-sepa\-ra\-bil\-ity of~$\mathfrak{I}$. As~in~the~proof of~Proposition~\ref{pe603}, we use induction on~the~length~$n$ of~the~sequence of~subgroups
$$
\mathfrak{I} = \mathfrak{I}_{0} \leqslant \mathfrak{I}_{1} \leqslant \ldots \leqslant \mathfrak{I}_{n} = X,
$$
where $\mathfrak{I}_{i+1}$ denotes the~normalizer of~$\mathfrak{I}_{i}$~in~$X$.

If $n = 1$, then $\mathfrak{I}$ is normal in~$X$, $X/\mathfrak{I}$ is residually a~$\mathcal{C}$\nobreakdash-group by~Proposition~\ref{pe701}, and~therefore $\mathfrak{I}$ is $\mathcal{C}$\nobreakdash-sepa\-ra\-ble in~$X$ by~Proposition~\ref{pe404}. Thus, we assume that $n \geqslant 2$, fix an~arbitrary element $x \in X \setminus \mathfrak{I}$, and~show that there exists a~subgroup $N \in \mathcal{C}^{*}(X)$ satisfying the~condition $x \notin \mathfrak{I}N$.

Since the~subgroup $\mathfrak{I} = \mathfrak{I}_{0}$ is $\mathfrak{P}(\mathcal{C})^{\prime}$\nobreakdash-iso\-lat\-ed in~$X$, then $\mathfrak{I}_{1}$ has the~same property by~Proposition~\ref{pe409}. Hence, if $x \notin \mathfrak{I}_{1}$, then, by~the~induction hypothesis, there exists a~subgroup $N \in \mathcal{C}^{*}(X)$ such that $x \notin \mathfrak{I}_{1}N$. It~follows that $x \notin \mathfrak{I}N$ and~therefore $N$ is the~required subgroup.

Let $x \in \mathfrak{I}_{1}$. Since $\mathfrak{I}$ is $\mathfrak{P}(\mathcal{C})^{\prime}$\nobreakdash-iso\-lat\-ed in~$\mathfrak{I}_{1}$, it follows from~the~above that there exists a~subgroup $M \in \mathcal{C}^{*}(\mathfrak{I}_{1})$ satisfying the~condition $x \notin \mathfrak{I}M$. By~applying Proposition~\ref{pe603} to~the~group~$X$ and~the~subgroups~$\mathfrak{I}_{1}$,~$M$, we find a~subgroup $N \in \mathcal{C}^{*}(X)$ such that $N \cap \mathfrak{I}_{1} \leqslant M$. It~is easy to~see that $\mathfrak{I}N \cap \mathfrak{I}_{1} \leqslant \mathfrak{I}M$ and~hence $x \notin \mathfrak{I}N$. Thus, $N$ is the~desired subgroup.
\end{proof}

\begin{proof}[\textup{\textbf{Proof of~Corollary~\ref{ce203}}}]
It follows from~Theorem~\ref{te202} that we only need to~prove the~necessity. Suppose that $X$ is a~tor\-sion-free nilpotent group with~the~property~$\mathcal{C}\mbox{-}\mathfrak{Sep}$,
\begin{equation}
1 = X_{0} \leqslant X_{1} \leqslant \ldots \leqslant X_{c} = X
\end{equation}
is its upper central series, and~$\mathcal{D}$ is the~class of~all periodic $\mathfrak{P}(\mathcal{C})$\nobreakdash-groups \pagebreak of~finite exponent. By~Proposition~\ref{pe402}, $\mathcal{C} \subseteq \mathcal{D}$, and~since $\mathfrak{P}(\mathcal{C}) = \mathfrak{P}(\mathcal{D})$, then $X$ has the~property $\mathcal{D}\mbox{-}\mathfrak{Sep}$. Let us show that all the~factors of~the~series~$(*)$ also have this property.

Indeed, suppose that $i \in \{0, \ldots, c-1\}$, $Y/X_{i}$ is a~$\mathfrak{P}(\mathcal{D})^{\prime}$\nobreakdash-iso\-lat\-ed subgroup of~$X_{i+1}/X_{i}$, and~$x \in X_{i+1} \setminus Y$. Since $X$ is tor\-sion-free, then it follows from~Proposition~\ref{pe409} that all the~factors of~the~series~$(*)$ are isolated in~$X$. Therefore, $Y$ is $\mathfrak{P}(\mathcal{D})^{\prime}$\nobreakdash-iso\-lat\-ed in~$X$ and,~by~the~property~$\mathcal{D}\mbox{-}\mathfrak{Sep}$, is $\mathcal{D}$\nobreakdash-sepa\-ra\-ble in~this group. Hence, there exists a~subgroup $N \in \mathcal{D}^{*}(X)$ satisfying the~condition $x \notin YN$. Then $xX_{i} \notin Y/X_{i} \cdot (NX_{i} \cap X_{i+1})/X_{i}$~and{\parfillskip=0pt
$$
NX_{i} \in \mathcal{D}^{*}(X), \quad 
NX_{i} \cap X_{i+1} \in \mathcal{D}^{*}(X_{i+1}), \quad 
(NX_{i} \cap X_{i+1})/X_{i} \in \mathcal{D}^{*}(X_{i+1}/X_{i})
$$
}because $\mathcal{D}$ is closed under taking subgroups and~quotient groups. Since $x$ is chosen arbitrarily, it follows that $Y/X_{i}$ is $\mathcal{D}$\nobreakdash-sepa\-ra\-ble in~$X_{i+1}/X_{i}$.

Thus, for~each $i \in \{0, \ldots, c-1\}$, the~quotient group~$X_{i+1}/X_{i}$ has the~property~$\mathcal{D}\mbox{-}\mathfrak{Sep}$ and,~by~Theorem~\ref{te201}, satisfies~$(3)_{\mathcal{D}}$. Since $\mathfrak{P}(\mathcal{C}) = \mathfrak{P}(\mathcal{D})$, the~conditions~$(3)_{\mathcal{C}}$ and~$(3)_{\mathcal{D}}$ are equivalent. As~noted above, $X_{i}$ is isolated in~$X$. Therefore, $X_{i+1}/X_{i}$ is tor\-sion-free and,~by~Proposition~\ref{pe501}, satisfies~$(5)_{\mathcal{C}}$. Therefore, $X \in \mathcal{C}\mbox{-}\mathcal{BN}$.
\end{proof}

\begin{proof}[\textup{\textbf{Proof of~Theorem~\ref{te204}}}]
1.\hspace{1ex}Suppose that
$$
x \in \mathcal{NC}\mbox{-}\mathfrak{Cl}(X, Y),\quad
V = \ker\sigma,\quad
U \in \mathcal{C}\mbox{-}\mathcal{BN}_{\mathfrak{P}(\mathcal{C})}^{*}(X),\quad\text{and}\quad
W = U \cap V.
$$
It~follows from~Proposition~\ref{pe502} that the~class $\mathcal{C}\mbox{-}\mathcal{BN}_{\mathfrak{P}(\mathcal{C})}$ is closed under taking subgroups and~finite direct products. So, $W \in \mathcal{C}\mbox{-}\mathcal{BN}_{\mathfrak{P}(\mathcal{C})}^{*}(X)$ by~Proposition~\ref{pe406}.

It is not difficult to~show, by~using Proposition~\ref{pe407}, that
$$
xV \in \mathcal{NC}\mbox{-}\mathfrak{Cl}(X/V,\,YV/V)
\quad\text{and}\quad
xW \in \mathcal{NC}\mbox{-}\mathfrak{Cl}(X/W,\,YW/W).
$$
It~follows from~the~nilpotency of~the~groups~$X/V$,~$X/W$ and~Theorem~\ref{te202} that the~equalities
\begin{gather*}
\mathcal{NC}\text{-}\mathfrak{Cl}(X/V,\,YV/V) = \mathcal{C}\text{-}\mathfrak{Cl}(X/V,\,YV/V) = \mathfrak{P}(\mathcal{C})^{\prime}\text{-}\mathfrak{Rt}(X/V,\,YV/V),\\
\mathcal{NC}\text{-}\mathfrak{Cl}(X/W,\,YW/W) = \mathcal{C}\text{-}\mathfrak{Cl}(X/W,\,YW/W) = \mathfrak{P}(\mathcal{C})^{\prime}\text{-}\mathfrak{Rt}(X/W,\,YW/W)
\end{gather*}
hold. Hence, there exist the~least $\mathfrak{P}(\mathcal{C})^{\prime}$\nobreakdash-num\-bers~$q$ and~$r$ satisfying the~conditions $(xV)^{q} \in\nolinebreak YV/V$ and~$(xW)^{r} \in YW/W$. Let $y,z \in Y$ be elements such that $yV = x^{q}V$ and~$zW = x^{r}W$. It~follows from~the~last equality and~the~inclusion $W \leqslant V$ that $(xV)^{r} = zV \in YV/V$. By~the~choice of~$q$ and~$r$, the~equality $r = qs$ holds for~some $\mathfrak{P}(\mathcal{C})^{\prime}$\nobreakdash-num\-ber~$s$. Hence, $y^{s}V = x^{qs}V = x^{r}V = zV$ and~therefore $y^{-s}z \in Y \cap V$. But~$Y \cap V = 1$, so $z = y^{s}$ and~$(xW)^{qs} = (xW)^{r} = zW = (yW)^{s}$.

Since the~$\mathcal{C}\mbox{-}\mathcal{BN}_{\mathfrak{P}(\mathcal{C})}$\nobreakdash-group~$X/W$ is nilpotent and~has no~$\mathfrak{P}(\mathcal{C})^{\prime}$\nobreakdash-tor\-sion, then, by~Proposition~~\ref{pe409}, the~extraction of~$\mathfrak{P}(\mathcal{C})^{\prime}$\nobreakdash-roots is unique in~this group. Therefore, $(xW)^{q} = yW$ and~$y^{-1}x^{q} \in W \leqslant U$. Since $U$ is chosen arbitrarily and~$X$ is residually a~$\mathcal{C}\mbox{-}\mathcal{BN}_{\mathfrak{P}(\mathcal{C})}$\nobreakdash-group, then $x^{q} = y \in Y$. Thus, the~subgroup $\mathcal{NC}\mbox{-}\mathfrak{Cl}(X, Y)$ coincides with~the~set $\mathfrak{P}(\mathcal{C})^{\prime}\mbox{-}\mathfrak{Rt}(X, Y)$ and~the~subgroup $\mathfrak{I} = \mathfrak{P}(\mathcal{C})^{\prime}\mbox{-}\mathfrak{Is}(X, Y)$. This means, in~particular, that $\mathfrak{I}$ is $\mathcal{NC}$\nobreakdash-sepa\-ra\-ble~in~$X$.

The~group~$X$ is residually a~$\mathcal{C}\mbox{-}\mathcal{BN}_{\mathfrak{P}(\mathcal{C})}$\nobreakdash-group, while a~$\mathcal{C}\mbox{-}\mathcal{BN}_{\mathfrak{P}(\mathcal{C})}$\nobreakdash-group is nilpotent and, by~Proposition~\ref{pe701}, is residually an~$\mathcal{NC}$\nobreakdash-group. Therefore, $X$ is residually an~$\mathcal{NC}$\nobreakdash-group and~so has no~$\mathfrak{P}(\mathcal{C})^{\prime}$\nobreakdash-tor\-sion in~accordance with~Proposition~\ref{pe405}. The~nilpotency classes of~$\mathcal{NC}\mbox{-}\mathfrak{Cl}(X, Y) = \mathfrak{I}$ and~$Y$ coincide by~Proposition~\ref{pe408}.
\smallskip

2.\hspace{1ex}Suppose again that $V = \ker\sigma$, $\mathfrak{I} = \mathfrak{P}(\mathcal{C})^{\prime}\mbox{-}\mathfrak{Is}(X, Y)$, and~$x \in \mathfrak{I} \cap V$. Since $\mathfrak{I} = \mathfrak{P}(\mathcal{C})^{\prime}\mbox{-}\mathfrak{Rt}(X, Y)$, then $x^{q} \in Y$ for~some $\mathfrak{P}(\mathcal{C})^{\prime}$\nobreakdash-num\-ber~$q$. It~follows that $x^{q} \in Y \cap V = 1$ and~$x = 1$ because $X$ has no~$\mathfrak{P}(\mathcal{C})^{\prime}$\nobreakdash-tor\-sion. Hence, $\mathfrak{I} \cap V = 1$, the~subgroup~$\mathfrak{I}$ is embedded into~the~$\mathcal{C}\mbox{-}\mathcal{BN}_{\mathfrak{P}(\mathcal{C})}$\nobreakdash-group~$X/V$ and~therefore itself belongs to~the~class~$\mathcal{C}\mbox{-}\mathcal{BN}_{\mathfrak{P}(\mathcal{C})}$.
\end{proof}

\section{Proof of~Theorems~\ref{te301}--\ref{te306}}\label{se08}

Let us use the~notation introduced in~Section~\ref{se03}.

\begin{eproposition}\label{pe801}
\textup{\cite[Theorem~2]{SokolovTumanova2016SMJ}}
Suppose that $\mathcal{C}$ is an~arbitrary root class of~groups, $P = \langle A * B;\ U \rangle$, and~$U$ is a~retract of~$B$. If~$A$ and~$B$ are residually $\mathcal{C}$\nobreakdash-groups, $U$ is $\mathcal{C}$\nobreakdash-sepa\-ra\-ble in~$A$, and~$A$ is $\mathcal{C}$\nobreakdash-quasi\-regular with~respect to~$U$, then $P$ is residually a~$\mathcal{C}$\nobreakdash-group.
\end{eproposition}

\begin{eproposition}\label{pe802}
\textup{\cite[Theorem~1]{SokolovTumanova2020IVM}}
Suppose that $\mathcal{C}$ is a~root class of~groups closed under taking quotient groups, $P = \langle A * B;\ U \rangle$, $A \ne U \ne B$, and~$U$ is normal in~$A$ and~$B$. Suppose also that
$$
\Omega = \big\{(R,S) \mid R \in \mathcal{C}^{*}(A),\ S \in \mathcal{C}^{*}(B),\ R \cap U = S \cap U\big\}.
$$
If~$\operatorname{Aut}_{P}(U)$ is abelian or~coincide with~$\operatorname{Aut}_{A}(U)$ or~$\operatorname{Aut}_{B}(U)$, then $P$ is residually a~$\mathcal{C}$\nobreakdash-group if and~only~if

\textup{1)}\hspace{1ex}$\bigcap_{(R,S) \in \Omega} R = 1 = \bigcap_{(R,S) \in \Omega} S$;

\textup{2)}\hspace{1ex}$U$ is $\mathcal{C}$\nobreakdash-sepa\-ra\-ble in~$A$ and~$B$.
\end{eproposition}

\begin{eproposition}\label{pe803}
\textup{\cite[Theorem~3]{Sokolov2022CA}}
Suppose that $\mathcal{C}$ is a~root class of~groups consisting only of~periodic groups and~closed under taking quotient groups, $G^{*} = \langle G, t;\ t^{-1}Ht = K,\ \varphi \rangle$, $K_{0} = G$, $H_{1} = H$, $K_{1} = K$, and~$H_{i+1} = H_{i} \cap K_{i}$, $K_{i+1} = H_{i+1}\varphi$ for~all $i \geqslant 1$. Suppose also that $G$ is residually a~$\mathcal{C}$\nobreakdash-group, $H$ and~$K$ are proper central subgroups of~$G$, and~$H_{n} = H_{n+1}$ for~some $n \geqslant 1$. If,~for~each $i \in \{0, 1, \ldots, n-1\}$, the~group~$K_{i}$ is $\mathcal{C}$\nobreakdash-regular with~respect to~$H_{i+1}K_{i+1}$ and~the~subgroup $\mathfrak{P}(\mathcal{C})^{\prime}\mbox{-}\mathfrak{Is}(K_{i}, H_{i+1}K_{i+1})$ is $\mathcal{C}$\nobreakdash-sepa\-ra\-ble in~$K_{i}$, then $G^{*}$ is residually a~$\mathcal{C}$\nobreakdash-group if and~only~if

\textup{1)}\hspace{1ex}$H_{n} = K_{n}$;

\textup{2)}\hspace{1ex}the subgroup $E = \operatorname{sgp}\{H_{n}, t\}$ is residually a~$\mathcal{C}$\nobreakdash-group;

\textup{3)}\hspace{1ex}$H$ and~$K$ are $\mathfrak{P}(\mathcal{C})^{\prime}$\nobreakdash-iso\-lat\-ed~in~$G$.
\end{eproposition}

\begin{eproposition}\label{pe804}
\textup{\cite[Theorem~3]{SokolovTumanova2017MZ}}
Suppose that $\mathcal{C}$ is a~root class of~groups consisting only of~periodic groups and~closed under taking quotient groups, $G^{*} = \langle G, t;\ t^{-1}Ht = K,\ \varphi \rangle$, $G$ is residually a~$\mathcal{C}$\nobreakdash-group, $H$ and~$K$ are proper central infinite cyclic subgroups of~$G$. If~$H \cap K \ne 1$, then $G^{*}$ is residually a~$\mathcal{C}$\nobreakdash-group if and~only~if

\textup{1)}\hspace{1ex}$H/H \cap K$ and~$K/H \cap K$ have the~same order;

\textup{2)}\hspace{1ex}$H$ and~$K$ are $\mathcal{C}$\nobreakdash-sepa\-ra\-ble~in~$G$;

\textup{3)}\hspace{1ex}$\mathcal{C}$ contains a~group of~order~$2$, unless $H \cap K$ lies in~the~center~of~$G^{*}$.
\end{eproposition}

\begin{eproposition}\label{pe805}
\textup{\cite[Theorem~3]{Sokolov2021SMJ2}}
Suppose that $\mathcal{C}$ is a~root class of~groups closed under taking quotient groups~and
$$
\mathcal{G}(\Gamma) = \big(\Gamma, \ G_{v}\ (v \in \mathcal{V}),\ H_{e}\ (e \in \mathcal{E}),\ \varphi_{\varepsilon e}\ (e \in \mathcal{E},\ \varepsilon = \pm 1)\big)
$$
is a~graph of~groups with~central trivially intersecting edge subgroups. If, for~each $v \in \mathcal{V}$, $G_{v}$~is residually a~$\mathcal{C}$\nobreakdash-group, $H_{v}$ is $\mathcal{C}$\nobreakdash-sepa\-ra\-ble in~$G_{v}$, and~the~latter is $\mathcal{C}$\nobreakdash-regular with~respect to~$H_{v}$, then $\pi_{1}(\mathcal{G}(\Gamma))$ is residually a~$\mathcal{C}$\nobreakdash-group.
\end{eproposition}

\begin{eproposition}\label{pe806}
\textup{\cite[Theorem~4]{Sokolov2021SMJ2}}
Suppose that $\mathcal{C}$ is a~root class of~groups closed under taking quotient groups,
$$
\mathcal{G}(\Gamma) = \big(\Gamma, \ G_{v}\ (v \in \mathcal{V}),\ H_{e}\ (e \in \mathcal{E}),\ \varphi_{\varepsilon e}\ (e \in \mathcal{E},\ \varepsilon = \pm 1)\big),
$$
$H_{e}\varphi_{\varepsilon e}$ is a~proper central subgroup of~$G_{e(\varepsilon)}$ for~all $e \in \mathcal{E}$, $\varepsilon = \pm 1$, and~at~least one of~the~following conditions hold:

\textup{1)}\hspace{1ex}$\mathcal{G}(\Gamma)$ is a~finite graph of~groups with~central trivially intersecting edge subgroups and,~for~each $v \in \mathcal{V}$, $G_{v}$~is $\mathcal{C}$\nobreakdash-regular with~respect~to~$H_{v}$;

\textup{2)}\hspace{1ex}$\Gamma$ is a~finite tree and,~for~any $e \in \mathcal{E}$, $\varepsilon = \pm 1$, $G_{e(\varepsilon)}$~is $\mathcal{C}$\nobreakdash-regular with~respect~to~$H_{e}\varphi_{\varepsilon e}$.\pagebreak

Then $\pi_{1}(\mathcal{G}(\Gamma))$ is residually a~$\mathcal{C}$\nobreakdash-group if and~only if all $G_{v}$ \textup{(}$v \in \mathcal{V}$\textup{)} are residually $\mathcal{C}$\nobreakdash-groups and,~for~any $e \in \mathcal{E}$, $\varepsilon = \pm 1$, $H_{e}\varphi_{\varepsilon e}$~is $\mathcal{C}$\nobreakdash-sepa\-ra\-ble~in~$G_{e(\varepsilon)}$.
\end{eproposition}

The~following proposition allows us to~omit the~requirement for~the~class of~groups to~be closed under taking quotient groups in~the~statements of~certain assertions.

\begin{eproposition}\label{pe807}
Suppose that we know a~necessary or~sufficient condition for~a~group~$X$ to~be residually a~$\mathcal{C}$\nobreakdash-group, where $\mathcal{C}$ is an~arbitrary root class of~groups consisting only of~periodic groups and~closed under taking quotient groups. If~the~parts of~this condition related to~$\mathcal{C}$ depend only on~the~set~$\mathfrak{P}(\mathcal{C})$ and~the~condition~$(5)_{\mathcal{C}}$, then the~same condition holds for~an~arbitrary root class consisting of~periodic groups.
\end{eproposition}

\begin{proof}
By~the~definition, a~class of~groups is root if and~only if it is closed under the~operations of~taking a~subgroup~($\mathbb{S}$), an~extension~($\mathbb{E}$), and~a~Cartesian degree~($\mathbb{D}$). In~addition to~them, we also consider the~operation of~taking a~quotient group~($\mathbb{F}$).

Suppose that $\mathcal{D}$ is the~class of~all periodic $\mathfrak{P}(\mathcal{C})$\nobreakdash-groups of~finite exponent, $\mathcal{SD}$ is the~class of~solvable $\mathcal{D}$\nobreakdash-groups, and~$\mathcal{C}_{1}$ is a~subclass of~$\mathcal{SD}$, in~which the~cardinality of~each group does not exceed the~cardinality of~some $\mathcal{C}$\nobreakdash-group (not~necessarily the~same for~all groups from~$\mathcal{C}_{1})$. Suppose also that $\mathcal{C}_{2}$ is the~class of~groups obtained from~$\mathcal{C}$\nobreakdash-groups by~a~finite number of~the~operations $\mathbb{S}$, $\mathbb{E}$, $\mathbb{D}$, and~$\mathbb{F}$. It~is clear that $\mathcal{C}_{2}$ is closed under the~listed operations. Since $\mathcal{C}$ is closed under $\mathbb{S}$, $\mathbb{E}$, and~$\mathbb{D}$, while $\mathcal{D}$ and~$\mathcal{SD}$ are also closed under~$\mathbb{F}$, which does not increase the~cardinality of~a~group, then $\mathcal{C}_{1}$ is closed under all four operations.

Obviously, $\mathcal{C} \subseteq \mathcal{C}_{2}$ and~$\mathfrak{P}(\mathcal{C}_{1}) = \mathfrak{P}(\mathcal{C})$. Since $\mathcal{D}$ is closed under $\mathbb{F}$ and~$\mathcal{C} \subseteq \mathcal{D}$ by~Proposition~\ref{pe402}, then $\mathcal{C} \subseteq \mathcal{C}_{2} \subseteq \mathcal{D}$ and~therefore $\mathfrak{P}(\mathcal{C}) \subseteq \mathfrak{P}(\mathcal{C}_{2}) \subseteq \mathfrak{P}(\mathcal{D}) = \mathfrak{P}(\mathcal{C})$. It~follows from~the~inclusion $\mathcal{C}_{1} \subseteq \mathcal{SD} \cap \mathcal{C}\mbox{-}\mathcal{BS}$ and~Proposition~\ref{pe601} that $\mathcal{C}_{1} \subseteq \mathcal{C}$. Thus, $\mathfrak{P}(\mathcal{C}_{1}) = \mathfrak{P}(\mathcal{C}) = \mathfrak{P}(\mathcal{C}_{2})$ and~$\mathcal{C}_{1} \subseteq \mathcal{C} \subseteq \mathcal{C}_{2}$, whence $(5)_{\mathcal{C}_{1}} \Rightarrow (5)_{\mathcal{C}} \Rightarrow (5)_{\mathcal{C}_{2}}$. Since the~operation~$\mathbb{F}$ does not increase the~cardinality, for~each $\mathcal{C}_{2}$\nobreakdash-group~$X$, there exists a~$\mathcal{C}$\nobreakdash-group~$Y$ whose cardinality is not less than the~cardinality of~$X$. Hence, $(5)_{\mathcal{C}} \Leftrightarrow (5)_{\mathcal{C}_{2}}$. Let us show that a~similar statement holds for~the~classes~$\mathcal{C}$~and~$\mathcal{C}_{1}$.

Suppose that $Y$ is a~$\mathcal{C}$\nobreakdash-group, $p \in \mathfrak{P}(\mathcal{C})$ (since $\mathcal{C}$ is root, then $\mathfrak{P}(\mathcal{C}) \ne \varnothing)$, $Z_{y}$ ($y \in Y$) is a~cyclic group of~order~$p$, and~$P$ is the~direct product of~the~groups~$Z_{y}$ ($y \in Y$). Then $Z_{y} \in \mathcal{C}$ in~accordance with~Proposition~\ref{pe401} and~$P \in \mathcal{C}$ by~the~definition of~a~root class. Hence, $P \in \mathcal{C}_{1}$ and,~moreover, the~cardinality of~$P$ is not less than the~cardinality of~$Y$. Therefore, $(5)_{\mathcal{C}_{1}} \Leftrightarrow (5)_{\mathcal{C}}$.

Thus, if $X$ is residually a~$\mathcal{C}$\nobreakdash-group, then it is residually a~$\mathcal{C}_{2}$\nobreakdash-group and~satisfies the~necessary condition (depending only on~$\mathfrak{P}(\mathcal{C})$ and~$(5)_{\mathcal{C}})$. If~the~sufficient condition holds (which also depends only on~$\mathfrak{P}(\mathcal{C})$ and~$(5)_{\mathcal{C}}$), then $X$ is residually a~$\mathcal{C}_{1}$\nobreakdash-group and~therefore is residually a~$\mathcal{C}$\nobreakdash-group.
\end{proof}

Proposition~\ref{pe807} allows us to~assume that, in~Theorems~\ref{te302}--\ref{te306}, the~class~$\mathcal{C}$ is closed under taking quotient groups. Therefore, Theorems~\ref{te305},~\ref{te306}, and~also Theorem~\ref{te301}, follow from~Theorem~\ref{te204} and~Propositions~\ref{pe603},~\ref{pe805},~\ref{pe806}, and~\ref{pe801}.

\begin{proof}[\textup{\textbf{Proof of~Theorem~\ref{te302}}}]
By~Theorem~\ref{te204}, $A$ and~$B$ are residually $\mathcal{C}$\nobreakdash-groups and~$U$ is $\mathcal{C}$\nobreakdash-sepa\-ra\-ble in~these groups if and~only if it is $\mathfrak{P}(\mathcal{C})^{\prime}$\nobreakdash-iso\-lat\-ed in~them. Let us show that, for~any subgroups $L \in \mathcal{C}^{*}(A)$ and~$M \in \mathcal{C}^{*}(B)$, one can find subgroups $R \in \mathcal{C}^{*}(A)$ and~$S \in \mathcal{C}^{*}(B)$ such that $R \leqslant L$, $S \leqslant M$, and~$R \cap U = S \cap U$. Then the~fact that $A$ and~$B$ are residually $\mathcal{C}$\nobreakdash-groups will mean that the~condition~1 of~Proposition~\ref{pe802} is satisfied. In~combination with~Proposition~\ref{pe807}, this proves the~required statement.

So, let $L \in \mathcal{C}^{*}(A)$, and~let $M \in \mathcal{C}^{*}(B)$. Because $\mathcal{C}$ is closed under taking subgroups, it follows from~the~relations $U/L \cap U \cong UL/L \leqslant A/L \in \mathcal{C}$ that $L \cap U \in \mathcal{C}^{*}(U)$. Similarly, $M \cap U \in \mathcal{C}^{*}(U)$, and,~by~Proposition~\ref{pe406}, $L \cap M \cap U \in \mathcal{C}^{*}(U)$. Let $q$ be the~exponent of~the~group $U/(L \cap M \cap U)$, which is finite by~Proposition~\ref{pe402}, and~let $N = U^{q}$. Then $U/N \in \mathcal{C}\mbox{-}\mathcal{BN} \cap \mathcal{C}$ in~accordance with~Propositions~\ref{pe502} and~\ref{pe601}. Since $N$ is normal in~$A$ and~$B$, and~these groups are $\mathcal{C}$\nobreakdash-regular with~respect to~$U$ by~Proposition~\ref{pe603}, then there exist subgroups $V \in \mathcal{C}^{*}(A)$ and~$W \in \mathcal{C}^{*}(B)$ satisfying the~equalities $V \cap U = N = W \cap U$. If~$R = V \cap L$ and~$S = W \cap M$, then $R \leqslant L$, $S \leqslant M$,
$$
R \cap U = V \cap L \cap U = N \cap L = N = N \cap M = W \cap M \cap U = S \cap U,
$$
and,~by~Proposition~\ref{pe406}, $R \in \mathcal{C}^{*}(A)$ and~$S \in \mathcal{C}^{*}(B)$. Therefore, $R$ and~$S$ are the~required subgroups.
\end{proof}

\begin{proof}[\textup{\textbf{Proof of~Theorem~\ref{te304}}}]
If $H \cap K = 1$, then the~conditions~1 and~3 are satisfied automatically and~the~statement to~be proved is a~special case of~Theorem~\ref{te306}. Let $H \cap K \ne 1$. By~Theorem~\ref{te204}, $G$~is residually a~$\mathcal{C}$\nobreakdash-group and~each subgroup lying in~$HK$ is $\mathcal{C}$\nobreakdash-sepa\-ra\-ble in~$G$ if and~only if it is $\mathfrak{P}(\mathcal{C})^{\prime}$\nobreakdash-iso\-lat\-ed in~this group. The~inclusion $2 \in \mathfrak{P}(\mathcal{C})$ holds if and~only if $\mathcal{C}$ contains a~group of~order~$2$. As~above, the~class~$\mathcal{C}$ can be considered closed under taking quotient groups. Therefore, the~required statement follows from~Proposition~\ref{pe804}.
\end{proof}

To~prove Theorem~\ref{te303}, we need the~following two assertions in~addition to~Proposition~\ref{pe803}.

\begin{eproposition}\label{pe808}
\textup{\cite[Proposition~5.7]{Sokolov2022CA}}
Let $\mathcal{C}$ be a~root class of~groups consisting only of~periodic groups and~closed under taking quotient groups. Suppose that $Y$ is an~abelian group, $\varphi \in \operatorname{Aut}Y$, and~$X$ is the~split extension of~$Y$ by~an~infinite cyclic group~$Z$ such that the~conjugation by~the~generator of~$Z$ acts on~$Y$ as~$\varphi$. Suppose also that $\Omega$ is a~family of~subgroups of~$Y$ defined as~follows: $N \in \Omega$ if and~only if $N \in \mathcal{C}^{*}(Y)$, $N\varphi = N$, and~the~order of~the~automorphism~$\varphi_{N}$ of~$Y/N$ induced by~$\varphi$ is finite and~is a~$\mathfrak{P}(\mathcal{C})$\nobreakdash-num\-ber. Then $X$ is residually a~$\mathcal{C}$\nobreakdash-group if and~only if $\bigcap_{N \in \Omega} N = 1$.
\end{eproposition}

\begin{eproposition}\label{pe809}
\textup{\cite[Proposition~6.8]{Sokolov2022CA}}
Suppose that $\mathcal{C}$ is a~root class of~groups consisting only of~periodic groups and~closed under taking quotient groups, $G^{*}\kern-1.5pt{} =\kern-1pt{} \langle G\kern-1pt{}, t;\ t^{-1}Ht\kern-1pt{} =\nolinebreak\kern-1pt{} K\kern-1pt{},\ \varphi \rangle$\kern-1pt{}, $K_{0} = G$, $H_{1} = H$, $K_{1} = K$, and~$H_{i+1} = H_{i} \cap K_{i}$, $K_{i+1} = H_{i+1}\varphi$ for~all $i \geqslant 1$. Suppose also that $H$ and~$K$ are proper central subgroups of~$G$, and~there exists $m \geqslant 1$ such that $H_{m}$ and~$K_{m}$ are finitely generated. If~$G^{*}$ is residually a~$\mathcal{C}$\nobreakdash-group, then $H_{n} = H_{n+1}$ for~some~$n$.
\end{eproposition}

\begin{proof}[\textup{\textbf{Proof of~Theorem~\ref{te303}}}]
It follows from~Proposition~\ref{pe807} that $\mathcal{C}$ can be considered closed under taking quotient groups. Let us show that the~conditions of~Proposition~\ref{pe803} hold.{\parfillskip=0pt\par}

By~Proposition~\ref{pe502}, $H_{i}, K_{i} \in \mathcal{C}\mbox{-}\mathcal{BA}$ for~all $i \geqslant 1$. Besides, $\sigma$ acts injectively on~$H_{1}K_{1}$. Therefore, it follows from~Proposition~\ref{pe603} and~Theorem~\ref{te204} that, for~any $i\kern-1pt{} \geqslant\kern-1pt{} 0$, the~group~$K_{i}$ is $\mathcal{C}$\nobreakdash-regular with~respect to~$H_{i+1}K_{i+1}$ and~the~$\mathfrak{P}(\mathcal{C})^{\prime}$\nobreakdash-iso\-la\-tor $\mathfrak{P}(\mathcal{C})^{\prime}\mbox{-}\mathfrak{Is}(K_{i}, H_{i+1}K_{i+1})$ is $\mathcal{C}$\nobreakdash-sepa\-ra\-ble in~$K_{i}$. The~fact that $G$ is residually a~$\mathcal{C}$\nobreakdash-group follows from~Theorem~\ref{te204} if $G$ is residually a~$\mathcal{C}\mbox{-}\mathcal{BN}_{\mathfrak{P}(\mathcal{C})}$\nobreakdash-group, and~from~the~assumption that $\mathcal{C}$ is closed under taking subgroups if $G^{*}$ is residually a~$\mathcal{C}$\nobreakdash-group.

If $H_{n} = K_{n}$ for~some $n > m$, then the~restriction of~$\varphi$ on~$H_{n}$ is an~automorphism of~the~latter and~therefore $E = \operatorname{sgp}\{H_{n}, t\}$ is a~split extension of~$H_{n}$ by~the~infinite cyclic subgroup~$\langle t \rangle$. If~$N \in \mathcal{C}^{*}(H_{n})$, then the~periodic $\mathcal{C}$\nobreakdash-group~$H_{n}/N$ is a~finite $\mathfrak{P}(\mathcal{C})$\nobreakdash-group because $H_{m}$ and~$K_{m}$ are finitely generated. Conversely, if $H_{n}/N$ is a~finite $\mathfrak{P}(\mathcal{C})$\nobreakdash-group, then $H_{n}/N \in \mathcal{C}$ in~accordance with~Proposition~\ref{pe401}. Therefore, by~Proposition~\ref{pe808}, the~condition~3 of~Theorem~\ref{te303} holds if and~only if $E = \operatorname{sgp}\{H_{n}, t\}$ is residually a~$\mathcal{C}$\nobreakdash-group.

Finally, we note that $H_{n} = H_{n+1}$ for~some $n > m$. This is obvious if $H_{n} = K_{n}$, and~follows from~Proposition~\ref{pe809} if $G^{*}$ is residually a~$\mathcal{C}$\nobreakdash-group (this proposition does not state that $n > m$, but it is clear that if $H_{n} = K_{n}$ for~some $n \geqslant 1$, then $H_{l} = K_{l}$ for~all~$l \geqslant n$). Thus, Theorem~\ref{te303} follows from~Proposition~\ref{pe803}.
\end{proof}

\section{The example}\label{se09}

Suppose that $\mathcal{C}$ is a~root class of~groups consisting only of~periodic groups. Let us show that if the~cardinalities of~all $\mathcal{C}$\nobreakdash-groups are less than the~cardinality of~some set~$\mathcal{I}$ (and~hence there are weakly $\mathcal{C}$\nobreakdash-bound\-ed abelian groups that are not $\mathcal{C}$\nobreakdash-bound\-ed), then there exists a~$\mathfrak{P}(\mathcal{C})^{\prime}$\nobreakdash-tor\-sion-free weakly $\mathcal{C}$\nobreakdash-bound\-ed nilpotent group, which is not residually a~$\mathcal{C}$\nobreakdash-group and~hence does not have the~property~$\mathcal{C}\mbox{-}\mathfrak{Sep}$.

\medskip

Since $\mathcal{C}$ contains non-trivial groups, then $\mathfrak{P}(\mathcal{C}) \ne \varnothing$. Suppose that $p \in \mathfrak{P}(\mathcal{C})$, $Y$~is the~direct product of~two cyclic groups of~order~$p$ with~generators~$y$,~$z$, and~$\chi$ is the~automorphism of~$Y$ defined by~the~rule $y\chi = yz$, $z\chi = z$. It~is easy to~see that the~order of~$\chi$ is equal to~$p$ and~therefore we can define the~split extension~$X$ of~$Y$ by~the~cyclic group of~order~$p$ with~a~generator~$x$, in~which the~conjugation by~$x$ acts on~$Y$ as~$\chi$. Let, for~each~$i \in \mathcal{I}$,
$$
X_{i}^{\vphantom{p}} = \big\langle x_{i}^{\vphantom{p}}, y_{i}^{\vphantom{p}}, z_{i}^{\vphantom{p}};\ x_{i}^{p} = y_{i}^{p} = z_{i}^{p} = [x_{i}^{\vphantom{p}}, z_{i}^{\vphantom{p}}] = [y_{i}^{\vphantom{p}}, z_{i}^{\vphantom{p}}] = 1,\ [y_{i}^{\vphantom{p}}, x_{i}^{\vphantom{p}}] = z_{i}^{\vphantom{p}} \big\rangle
$$
be an~isomorphic copy of~$X$, and~let $\Gamma$ be the~star graph with~a~central vertex~$w$ and~the~set~$\mathcal{I}$ as~a~set of~leaves. We define a~graph of~groups~$\mathcal{G}(\Gamma)$ by~associating the~vertex~$w$ with~the~group $Z = \langle z;\ z^{p} = 1 \rangle$, the~vertex $i \in \mathcal{I}$ with~the~group~$X_{i}$, and~the~edge connecting~$w$ and~$i$ with~the~group~$Z$ and~the~homomorphisms acting according to~the~rules: $z \mapsto z$, $z \mapsto z_{i}$.

Let $D$ be the~generalized direct product associated with~$\mathcal{G}(\Gamma)$, i.e.,~the~quotient group of~the~direct product of~the~groups~$X_{i}$ ($i \in \mathcal{I}$) and~$Z$ by~the~normal closure of~the~set $\{z_{i}z^{-1} \mid i \in \mathcal{I}\}$. Since $\Gamma$ is a~tree, then $Z$ is embedded into~$D$ by~the~identity mapping of~the~generator~\cite[Theorem~1]{SokolovTumanova2019AL}. Therefore, the~relation $z \ne 1$ holds~in~$D$.

It is easy to~see that the~presentation of~$D$ can be reduced to~the~form
\begin{multline*}
\big\langle x_{i}^{\vphantom{p}}, y_{i}^{\vphantom{p}}, z\ (i \in \mathcal{I});\ x_{i}^{p} = y_{i}^{p} = z^{p} = 1,\ [x_{i}^{\vphantom{p}}, z] = [y_{i}^{\vphantom{p}}, z] = 1,\ [y_{i}^{\vphantom{p}}, x_{i}^{\vphantom{p}}] = z\ (i \in \mathcal{I}),\\
[x_{k}^{\vphantom{p}}, x_{l}^{\vphantom{p}}] = [y_{k}^{\vphantom{p}}, y_{l}^{\vphantom{p}}] = [x_{k}^{\vphantom{p}}, y_{l}^{\vphantom{p}}] = 1\ (k, l \in \mathcal{I},\ k \ne l) \big\rangle,
\end{multline*}
which implies that the~sequence $1 \leqslant Z \leqslant D$ is a~central series of~$D$ with~factors of~exponent~$p$ (here $Z$ is still the~subgroup generated by~$z$). It~is clear that $D$ has no~$\mathfrak{P}(\mathcal{C})^{\prime}$\nobreakdash-tor\-sion and~the~factors of~the~specified series are weakly $\mathcal{C}$\nobreakdash-bound\-ed, i.e.,~$D \in \mathcal{C}\mbox{-}w\mathcal{BN}$. Let us show, however, that $D$ is not residually a~$\mathcal{C}$\nobreakdash-group and~therefore it is the~required group.

Indeed, let $\sigma$ be a~homomorphism of~$D$ onto~some $\mathcal{C}$\nobreakdash-group~$T$. Then the~cardinality of~$\mathcal{I}$ exceeds the~cardinality of~$T$ and~hence there exist $k, l \in \mathcal{I}$ such that $k \ne l$ and~$x_{k}\sigma = x_{l}\sigma$. It~follows that $z\sigma = [y_{k}\sigma, x_{k}\sigma] = [y_{k}\sigma, x_{l}\sigma] = 1$, and~since $\sigma$ is chosen arbitrarily, $D$ is not residually a~$\mathcal{C}$\nobreakdash-group.

\end{document}